\documentclass[12pt,reqno]{amsart}
\pdfoutput=1
\usepackage{cases}
\usepackage{amscd}
\usepackage{amsfonts}
\usepackage{amssymb}
\usepackage{amsmath}
\usepackage{graphicx}
\usepackage{epstopdf}
\usepackage{subfigure}
\usepackage{amsthm}
\usepackage{stmaryrd}

\theoremstyle{remark}
\newtheorem{example}{\textbf{Example}}[section]
\numberwithin{equation}{section}

\usepackage{color}
\usepackage{datetime}

\def\tr{\textcolor{black}}
\def\tb{\textcolor{black}}

\makeatletter
\newcommand\figcaption{\def\@captype{figure}\caption}
\newcommand\tabcaption{\def\@captype{table}\caption}
\makeatother

\usepackage{geometry}
\usepackage{algorithm}
\usepackage{multirow}

\def\bq{\begin{equation}}
\def\eq{\end{equation}}
\def\bqs{\begin{equation*}}
\def\eqs{\end{equation*}}
\def\bsqs{\begin{subequations}}
\def\esqs{\end{subequations}}
\def\ba{\begin{aligned}}
\def\ea{\end{aligned}}
\def\br{\begin{eqnarray}}
\def\er{\end{eqnarray}}
\def\brr{\bq\begin{array}{rlll}}
\def\err{\end{array}\eq}

\def\text#1{\hbox{#1}}
\newtheorem{thm}{Theorem}[section]

\newtheorem{rem}{Remark}[section]

\newcommand{\bsub}{\begin{subequations}}
\newcommand{\esub}{\end{subequations}$\!$}

\title[mixed DG without interior penalty]{Unconditionally energy stable DG schemes for the Swift--Hohenberg equation}

\author[H.~Liu, P. ~Yin]{Hailiang Liu$^\dagger$ and Peimeng Yin$^\dagger$}
\email{hliu@iastate.edu; pemyin@iastate.edu}
\address{$^\ddagger$ Iowa State University, Mathematics Department, Ames, IA 50011} 
\keywords{Swift-Hohenberg equation, energy stability, DG method, implicit-explicit time stepping}
\subjclass{65N12, 65N30,  35K35}

\begin{document}

\begin{abstract}
The Swift--Hohenberg equation as a central nonlinear model in modern physics has a gradient flow structure.  Here we introduce fully discrete discontinuous Galerkin (DG) schemes for a class of fourth order gradient flow problems, including the  nonlinear Swift--Hohenberg equation, to produce
free-energy-decaying discrete solutions, irrespective of the time step and the mesh size.
We exploit and extend the mixed DG method  introduced in [H. Liu and P. Yin, J. Sci. Comput., 77: 467--501, 2018] for the spatial discretization, and
the ``Invariant Energy Quadratization" method for the time discretization.
The resulting IEQ-DG algorithms are linear, thus they can be efficiently solved without resorting to any iteration method. We actually prove that these schemes are unconditionally energy stable.  We present several numerical examples that support our theoretical results and illustrate the efficiency, accuracy and energy stability of our new algorithm. 
The numerical results on two dimensional pattern formation problems indicate that the method is able to deliver comparable patterns of high accuracy.
\end{abstract}

\maketitle

\bigskip



\section{Introduction}

Motivated by fluid mechanics, reaction-diffusion chemistry, and biological systems, pattern forming nonequilibrium systems continue to attract significant research interest (see e.g. \cite{Ho06, CG09}). They form a broad class of dissipative nonlinear partial differential equations (PDEs) that describe important processes in nature. These PDEs,
such as the Swift--Hohenberg (SH) equation \cite{SH77} and  extended Fisher--Kolmogorov equations \cite{DS88, PT95},  generally cannot be solved analytically. Therefore, computer simulations play an essential role in understanding of the non-equalibrium processing and how it leads to pattern formation.

We consider the following model equation
\begin{align}\label{fourthPDE}
u_t  = -\Delta^2 u -a \Delta u -\Psi'(u), \; x\in \Omega \subset \mathbb{R}^d, \; t>0,
\end{align}
where $u(x, t)$  is a scalar time-dependent unknown defined in $\Omega$, a spatial domain of $d$ dimension, and $\Psi$ is a given nonlinear function. Here the model parameter $a$ is a constant.
This falls into the large class of relaxation models forming stable patterns studied in \cite{FK99}.  Throughout this work we assume that
\begin{align}\label{pa}
\Phi(w):=\Psi(w) -\frac{a^2}{8}w^2  \; \text{is bounded from below},
\end{align}
and {the domain boundary $ \partial \Omega$ has a  unit outward normal $\nu$}.
We consider the initial/boundary value problem for (\ref{fourthPDE}) with initial data
$
u(x, 0)=u_0(x),
$
subject to either periodic boundary conditions, or homogenous boundary conditions such as
\begin{equation}\label{nonperiodic}
(i) \ u=\partial_\nu u=0;  \quad (ii) \ u=\Delta u=0; \quad (iii) \ \partial_\nu u=\partial_\nu \Delta u=0, \quad x\in \partial \Omega,\; t>0.
\end{equation}
Thus equation (\ref{fourthPDE})  may be written as
$$
u_t = -\frac{\delta \mathcal{E}}{\delta u},
$$
where $\frac{\delta \mathcal{E}}{\delta u}$ is the $L^2$ variational derivative, 
and  $\mathcal{E}$ is the free energy functional (or Lyapunov functional)
$$
\mathcal{E}(u) = \int_\Omega \frac{1}{2}\left( \Delta u+\frac{a}{2}u\right)^2 + \Phi(u) dx.
$$
One can show that at least for classical solutions,
\bq\label{engdisO}
\frac{d}{dt}\mathcal{E}(u) =  -\int_{\Omega}|u_t|^2dx \leq 0.
\eq
With assumption (\ref{pa}), the free energy $\mathcal{E}$ is bounded from below, hence convergence to steady states is expected as $t\to +\infty$. 
The expression (\ref{engdisO}) as a fundamental property of (\ref{fourthPDE}) is naturally desired for high order numerical approximations.  The objective of this paper is to develop high order discontinuous Galerkin (DG)  schemes 
 which inherit this property for arbitrary meshes and time step sizes. 
We note that assumption (\ref{pa}) will be essentially used in our time discretization. 

This study is motivated by the Swift--Hohenberg equation in the theory of pattern formation,
\begin{align}\label{SH}
u_t= \epsilon u -(\Delta+1)^2 u +gu^2 -u^3,
\end{align}
where  $\epsilon $ and $g$ are  physical parameters. Such model was derived by J. Swift and P. C. Hohenberg \cite{SH77} to describe Rayleigh-B\'{e}nard convection \cite{EMS10, WDLCD12}. Related applications can be found in complex pattern formation \cite{KS12}, complex fluids and biological tissues \cite{HA05}.
The Swift--Hohenberg equation is also known to have many qualitatively different equilibrium solutions such as two-dimensional
 quasipatterns \cite{BIS17}, and the pattern selection can depend on parameters $\epsilon, g$ and the size of the domain; see e.g. \cite{BPT01, PR04}.


The Swift--Hohenberg equation is a gradient flow and requires very long time simulations to reach steady states.
From the numerical perspective, an ideal scheme to solve a gradient flow would  (i) preserve the energy dissipation, (ii) be more accurate,  (iii) be efficient, and, (iv) perhaps above all, be simple to implement.  Among these the first aspect is particularly important, and is crucial to eliminating numerical results that are not physical (see e.g.  \cite{CPWV97, CP02} ). For the Swift--Hohenberg equation, an explicit time discretization is known to require a time step extremely small to preserve the energy dissipation (see e.g. \cite{XGV91}). Several numerical methods have been developed to alleviate the time step restriction while still keeping the energy dissipation, related contributions include the fully implicit operator splitting finite difference method  \cite{CPWV97, CP02}, 
the semi-analytical Fourier spectral method \cite{Le17}, 
the unconditionally energy stable method \cite{GN12} derived from an integration quadrature formula, the large time-stepping method \cite{ZM16} based on the use of an extra artificial stabilized term, and the energy stable generalized-$\alpha$ method \cite{SEVDPC17}. However, these methods generally require the use of an iteration in solving the fully discrete nonlinear systems.  We report here on a new method which seems to be promising.
{Our numerical results will be on one and two-dimensional cases. The relevant application is indeed mostly in two dimensional space, although some three dimensional versions of the model also describe interesting patterns, see e.g. \cite{TARGK13}.
}

For the spatial discretization,  we exploit and further extend the mixed discontinuous Galerkin method introduced in \cite{LY18}. The method involves three ingredients: (a) rewriting the scalar equation into a symmetric system called mixed formulation; (b) applying the DG discretization to the mixed formulation using only central fluxes on interior cell interfaces; and (c) weakly enforcement of boundary conditions of types as listed in (\ref{nonperiodic}) through both $u$ and the auxiliary variable $q=-\left(\Delta+\frac{a}{2}\right)u$.  For periodic boundary conditions and quadratic $\Psi$,  both $L^2$ stability and optimal $L^2$ error estimates  of the resulting semi-discrete DG method have been established in \cite{LY18}  for  both one dimensional and two dimensional cases using tensor-product polynomials on rectangular meshes.

In this work,  we show that the mixed DG discretization can be refined into a unified form that works for all homogeneous boundary conditions, and further show it satisfies the energy dissipation law (\ref{engdisO}) with a discrete energy of form $\mathcal{E}(u_h, q_h)=\int_{\Omega} (\frac{1}{2}|q_h|^2 +\Phi(u_h)) dx$. {Note that due to the weakly enforcement of boundary condition (i) in (\ref{nonperiodic}), the corresponding discrete energy requires a correction term (vanishing when mesh is refined) so that a discrete energy dissipation law is ensured. }

Our mixed DG method has the usual advantages of a DG method (see e.g. \cite{HW07, Ri08, Sh09}) over the continuous Galerkin methods, such as high order accuracy, flexibility in hp-adaptation, capacity to handle domains with complex geometry,  its distinctive feature lies in numerical  flux choices without using any interior penalty.  For more references to earlier results on DG numerical approximations of some fourth order PDEs, we refer to \cite{LY18}.

For the temporal discretization, instead of using the method studied in \cite{LY18} which requires iteratively solving a nonlinear system, we explore the method of \emph{Invariant Energy Quadratization} (IEQ), which was proposed very recently in \cite{Y16, ZWY17}. This method is a generalization of the method of Lagrange multipliers or of auxiliary variables originally proposed in \cite{BGG11, GT13}.
With this method, we introduce an auxiliary variable $U=\sqrt{\Phi+B}$, where $\Phi(u)+B > 0$ for some constant $B>0$, so that
$$
\Phi'(u)=H(u) U, \quad  U_t = \frac{1}{2} H(u)u_t,
$$
where  $H(u):=\Phi'(u)/\sqrt{\Phi(u)+B}$.  Such method when applied to the semi-discrete DG formulation requires only replacing the nonlinear function $\Phi'(u_h^{n+1})$ by $H(u^{n}_h)U^{n+1}$, where $u^{n}_h$ is the approximation of $u_h$ in the previous time step. $U^{n+1}$ is updated from $U^n$ in two steps:  the piecewise $L^2$ projection with $U_h^n=\Pi U^n$, and the update step with
$$
\frac{U^{n+1} - U_h^n}{\Delta t}  = \frac{1}{2} H(u^{n}_h) \frac{u_h^{n+1} - u_h^n}{\Delta t}.
$$
This treatment when coupled with the DG discretization described above leads to
\begin{subequations}\label{FPDGFull1st++}
	\begin{align}
	\left(  \frac{u_h^{n+1} - u_h^n}{\Delta t}, \phi \right)  = & - A(\phi, q_h^{n+1})-\left( H(u_h^n)U^{n+1}, \phi \right), \\
(q_h^{n},\psi) = & A(u_h^{n}, \psi),
	\end{align}
\end{subequations}
for $\forall \phi, \psi$ in the space of piecewise polynomials,  $A(\cdot, \cdot)$ is a bilinear operator corresponding to the operator $\mathcal{L}=-\left(\Delta+\frac{a}{2}\right)$.
To obtain a second order in time discretization, we replace
$q_h^{n+1}$ and $U^{n+1}$ in (\ref{FPDGFull1st++}a) by $(q_h^{n+1}+q_h^n)/2$ and $(U^{n+1}+U_h^n)/2$, respectively, and  replace $H(u_h^n)$ by $H(u^{n, *}_h)$,  with $u^{n, *}_h= \frac{3}{2}u_h^n-\frac{1}{2}u_h^{n-1}$.  We prove that these schemes are unconditionally energy stable.
In addition, the resulting discrete systems are linear with scale comparable to that generated by the same DG discretization to the linear problem. As a result, the methods are simple to implement and computationally efficient to achieve high order of accuracy in space.


This paper is organized as follows: in Section 2, we formulate a unified semi-discrete DG method for  (\ref{fourthPDE}) subject to  different boundary conditions. In Section 3, we present first order and second order fully discrete DG schemes and show their energy dissipation properties. In Section 4, we first present numerical results to demonstrate the high order of accuracy of the proposed schemes, and their energy dissipating property, and we further simulate some two dimensional pattern formation problems, including two particular patterns, rolls and hexagons, arising during the Rayleigh-B\'{e}nard convection as simulated in \cite{PCC14, DA17}.  Finally in Section 5 some concluding remarks are given.


\section{Symmetrization and spatial discretization }
 In this section we recall the mixed DG spatial discretization introduced in \cite{LY18} and show it also satisfies the energy dissipation law for the nonlinear problem (\ref{fourthPDE}) when subjected to homogeneous boundary conditions.
\subsection{Symmetrization}  The idea in \cite{LY18} is to apply the mixed DG discretization without interior penalty to a symmetrized mixed formulation. For the fourth order PDE  (\ref{fourthPDE}),  we let $\mathcal{L}=-\left(\Delta+\frac{a}{2}\right)$ so that the model admits the following form
$$
u_t = -\mathcal{L}^2u-\Phi'(u).
$$
Further set $q=\mathcal{L}u$, then
\begin{equation}\label{mix}
\left \{
\begin{array}{rl}
    u_t = &- \mathcal{L} q-\Phi'(u),\\
    q = & \mathcal{L} u.
\end{array}
\right.
\end{equation}
Let $V_h$ denote the discontinuous Galerkin finite element space, then the DG method for (\ref{mix}) is to find $(u_h(\cdot, t), q_h(\cdot, t))\in V_h\times V_h  $ such that
\begin{subequations}\label{SemiDG}
\begin{align} 
& (u_{ht}, \phi) = -A(q_h, \phi) - (\Phi'(u_h), \phi), \\
& (q_h, \psi)=A(u_h, \psi),
\end{align}
\end{subequations}
for all $\phi, \ \psi \in V_h$.  Here $A(q_h, \phi)$ is the DG discretization of $(\mathcal{L} q, \phi)$ and $A(u_h, \psi)$ is the DG discretization
of  $(\mathcal{L} u, \psi)$. The precise form of $A(\cdot, \cdot)$ will be given in the next subsection depending on the types of boundary conditions.
The initial data for $u_h$ is taken as $u_h(x,0)=\Pi u_0(x)$, here $\Pi$ is the piecewise $L^2$ projection, more precisely $u_h(x,0) \in V_h$ satisfying
\bqs
\int_{\Omega} (u_0(x)-u_h(x,0))\phi dx=0, \quad \forall \phi\in V_h.
\eqs
We should point out that  the advantages of symmetry in the scheme formulation lie at least in two aspects: (i) unconditional energy stability of the semi-discrete scheme, and (ii) easy computation since the resulting discrete system has a symmetric  coefficient matrix.
\subsection{DG discretization} The mixed semi-discrete DG scheme (\ref{SemiDG}) was presented in \cite{LY18} for one and two dimensional rectangular meshes. Here we extend it to a unified form valid for more general meshes and different boundary conditions, and further study its energy dissipation property.

To extend the results in \cite{LY18} to general meshes we need to recall some conventions.  Let the domain $\Omega$ be a union of  shape regular meshes $\mathcal{T}_h=\{K\}$, with the mesh size $h_K = \text{diam}\{K\}$ and $h=\max_{K} h_K$.  We denote the set of the interior interfaces by $\Gamma^0$, and  the set of all boundary faces by $\Gamma^\partial$.  Then the discontinuous Galerkin finite element space can be formulated as
$$
V_h = \{v\in L^2(\Omega) \ : \ v|_{K} \in P^k(K), \ \forall K \in \mathcal{T}_h \},
$$
where $P^k(K)$ denotes the set of polynomials of degree no more than $k$ on element $K$. If the normal vector on the element interface $e\in \partial K_1 \cap \partial K_2$ is oriented from $K_1$ to $K_2$, then the average $\{\cdot\}$ and the jump $[\cdot]$ operator are defined by
$$
\{v\} = \frac{1}{2}(v|_{\partial K_1}+v|_{\partial K_2}), \quad [v]=v|_{\partial K_2}-v|_{\partial K_1}, 
$$
for any function $v \in V_h$,  where $v|_{\partial K_i} \ (i=1,2)$ is the trace of $v$ on $e$ evaluated from element $K_i$.


The direct DG discretization of (\ref{mix}), following \cite{LY18}, is of the form
\begin{subequations}\label{SemiDGscell}
\begin{align}
\int_K u_{ht}  \phi dx = &- \int_K \nabla q_h \cdot \nabla \phi dx +  \int_{\partial K} \widehat{\partial_\nu q_h} \phi + (q_h-\widehat{q_h})\partial_\nu \phi ds + \int_K \left( \frac{a}{2} q_h- \Phi'(u_h)\right) \phi dx,\\
\int_K q_h  \psi dx = & \int_K \nabla u_h \cdot \nabla \psi dx -   \int_{\partial K} \widehat{\partial_\nu u_h} \psi + (u_h-\widehat{u_h})\partial_\nu \psi ds - \int_K \frac{a}{2} u_h \psi dx,
\end{align}
\end{subequations}
for  $u_h, \ q_h \in V_h$ with test functions  $\phi, \ \psi \in V_h$.  \tr{Here with a slight abuse of notation,  we use $\nu$ to also stand for the outward normal direction to $\partial K$ for each $K$}.  On cell interfaces $e \in \partial K \bigcap \Gamma^0$,  central numerical fluxes
\begin{equation}\label{fluxI}
\widehat{\partial_\nu q_h} = \{\partial_\nu  q_h\}, \ \widehat{q_h} = \{q_h\}, \ \widehat{\partial_\nu u_h} = \{\partial_\nu  u_h\}, \ \widehat{u_h} = \{u_h\}
\end{equation}
are adopted in \cite{LY18}.  Boundary fluxes on $e \in \partial K \bigcap \Gamma^\partial$ depend on boundary conditions pre-specified.  For periodic boundary conditions, the numerical fluxes  can take the same formula as those in (\ref{fluxI}).
For non-homogeneous boundary conditions
\begin{align}\label{3bd}
(i) \; u=g_1, \partial_\nu u=g_2;  \quad (ii)\; u= g_1, \Delta u=g_3; \quad (iii)\; \partial_\nu u=g_2, \partial_\nu \Delta u=g_4 \quad \text{on} \;  \partial \Omega,\; t>0,
\end{align}
the boundary fluxes introduced in \cite{LY18} are respectively defined by
 \begin{align}\label{fluxBD1}
  \widehat{u_h} = g_1, \;
\widehat{\partial_\nu u_h} =  g_2,
  \widehat{q_h} =  q_h, \;
 \widehat{\partial_\nu q_h} = \frac{\beta_1}{h} (g_1-u_h)+ \partial_\nu q_h;
\end{align}
  \begin{align}\label{fluxBD2}
 \widehat{u_h} = g_1, \;
\widehat{\partial_\nu u_h} =\frac{\beta_0}{h} (g_1-u_h)+ \partial_\nu u_h;
  \widehat{q_h} =  -g_3-\frac{a}{2}g_1, \;
 \widehat{\partial_\nu q_h} = \frac{\beta_0}{h} (-g_3-\frac{a}{2}g_1-q_h) + \partial_\nu q_h;
\end{align}
 \begin{align}\label{fluxBD3}
& \widehat{u_h} = u_h, \;
\widehat{\partial_\nu u_h} =  g_2; \quad
  \widehat{q_h} =  q_h, \;
 \widehat{\partial_\nu q_h} =-g_4-\frac{a}{2}g_2,
 \end{align}
where the flux parameters $\beta_{0}, \ \beta_{1}$ are used to weakly enforce the specified boundary conditions. { Note that
$h$ in $\frac{\beta_0}{h}$ or   $\frac{\beta_1}{h}$ needs to be carefully chosen when using unstructured meshes. In practice, it has been selected as the distance from cell center to the domain boundary. }

Summation of (\ref{SemiDGscell}) over all elements $K\in \mathcal{T}_h$ leads to a unified DG formulation
\begin{subequations}\label{SemiDGN}
\begin{align}
(u_{ht},\phi) +\alpha h^{-1} (u_h, \phi)_{\Gamma^{\partial}}= &- A(\phi, q_h) -\left(\Phi'(u_h),\phi \right),\\
(q_h, \psi) = &  A(u_h,\psi),
\end{align}
\end{subequations}
for periodic and homogeneous boundary conditions, i.e. $g_i=0$. Here
the bilinear functional
$$
A(w,v)=A^0(w,v)+A^b(w,v)
$$
with
\bq \label{A0}
A^0(w,v)= \sum_{K\in \mathcal{T}_h} \int_K \left( \nabla w \cdot \nabla v  - \frac{a}{2} w v \right)dx + \sum_{e\in \Gamma^0} \int_e \left( \{\partial_\nu w\}[v]+ [w]\{\partial_\nu v\} \right)ds.
\eq
Both the method parameter $\alpha$ and  $A^b(\cdot, \cdot)$ are given below for each respective type of boundary conditions:
\begin{subequations}\label{bd+}
\begin{align}
\text{for periodic case} \qquad & \alpha=0, A^b(w, v)=\frac{1}{2}  \int_{\Gamma^{\partial} } \left( \{\partial_\nu w\}[v]+ [w]\{\partial_\nu v\} \right)ds, \\
\text{for (i)}   \qquad & \alpha=\beta_1, \; A^b(w, v)= - \int_{\Gamma^{\partial}}   w \partial_\nu v ds, \\
\text{for (ii)} \qquad  & \alpha=0,   \; A^b(w,v)=  \int_{\Gamma^{\partial}} \frac{\beta_0}{h} wv - w\partial_\nu v -\partial_\nu w v  ds,\\
\text{for (iii)} \qquad  &  \alpha=0, \; A^b(w,v)=0.
\end{align}
\end{subequations}
Note that for periodic case in (\ref{bd+}a) the left boundary and the right boundary are considered as same boundaries, for which we use the factor $1/2$
to avoid the recounting.
\begin{rem}
For case (i),  $\alpha \not=0$ and $A^b(\cdot, \cdot)$ is non-symmetric;  our numerical results indicate that,
the optimal order of  accuracy may not be obtained if $\alpha=0$ in such case. For other types of boundary conditions, $\alpha=0$ and
$A^b(\cdot, \cdot)$ is symmetric, hence (\ref{SemiDGN}) reduces to (\ref{SemiDG}).
\end{rem}

\subsection{Energy stability of the DG scheme}
For the semi-discrete DG scheme (\ref{SemiDGN}), we have the following energy dissipation property.
\begin{thm}
The semi-discrete DG scheme (\ref{SemiDGN}) with $\alpha \geq 0$ satisfies a discrete energy dissipation law
\bqs
\frac{d}{dt}\mathcal E(u_h,q_h) = -\int_{\Omega}|u_{ht}|^2dx \leq 0,
\eqs
where
\bq\label{oeq}
\mathcal E(u_h,q_h)=\int_{\Omega} \frac{1}{2}|q_h|^2 + \Phi(u_h) dx+\frac{\alpha}{2 h}\int_{\Gamma^{\partial}} u_h^2ds.
\eq
\end{thm}
\begin{proof} Taking $\phi=u_{ht}$ in (\ref{SemiDGN}a), and $\psi=q_h$ {in
$$
(q_{ht}, \psi) =  A(u_{ht},\psi),
$$
which is a resulting equation from differentiation of (\ref{SemiDGN}b) in $t$}, upon summation, we obtain
the desired result.
\end{proof}
\begin{rem} {For case (i) with $\alpha \not=0$, the discrete energy $\mathcal E(u_h,q_h)$ is still consistent with the free energy at the
continuous level. To see this, we can informally argue by assuming that $\|u_h-g\|_{L^\infty(\partial \Omega)} \sim  h^{k+1}$, which is the order of accuracy when using polynomials of degree $k$, then with uniform meshes and note that $g=0$, we have
 $$
\frac{\alpha}{2h} \left|
\int_{\partial \Omega} u_h^2 dx
\right| \sim \frac{1}{2} \alpha |\partial \Omega|h^{2k+1},
$$
which tends to vanish as $h \to 0$.}
\end{rem}

\subsection{Non-homogeneous boundary conditions}
For non-homogeneous boundary conditions (i)-(iii) in (\ref{3bd}), the unified DG scheme (\ref{SemiDGN})
becomes
\begin{subequations}\label{SemiDGNg}
\begin{align}
(u_{ht},\phi) +\alpha h^{-1} (u_h, \phi)_{\Gamma^{\partial}} = & - A(\phi, q_h) -\left(\Phi'(u_h),\phi \right) + L_1(t;\phi),\\
(q_h, \psi) = & A(u_h,\psi)+L_2(t;\psi),
\end{align}
\end{subequations}
where  $L_i(t; \cdot), i=1, 2$ are given below for each respective type of boundary conditions:
\begin{subequations}\label{bd}
\begin{align}
\text{for (i)}   \qquad & L_1(t;v) = \int_{\Gamma^{\partial}} \frac{\beta_1}{h} g_1v ds, \\
& L_2(t;v) =  \int_{\Gamma^{\partial}}  \left(g_1 \partial_\nu v - g_2v \right)ds;\\
\text{for (ii)} \qquad  & L_1(t;v) = \int_{\Gamma^{\partial}}
\left((g_3+ag_1/2) \partial_\nu v - \frac{\beta_0}{h} (g_3+ag_1/2)v \right) ds,\\
& L_2(t;v) =  \int_{\Gamma^{\partial}} \left(- g_1 \partial_\nu v -\frac{\beta_0}{h} g_1v\right) ds;\\
\text{for (iii)} \qquad  &  L_1(t;v) = -\int_{\Gamma^{\partial}} (g_4+ag_2/2) vds,\\
& L_2(t;v) = - \int_{\Gamma^{\partial}}  g_2 vds.
\end{align}
\end{subequations}
The dependence of $L_i(t; \cdot)$ on $t$ comes from the fact that $g_i(i=1, \cdots, 4)$ are functions of $x$ and $t$. The choices for parameters $\beta_0$ and $\beta_1$ have been discussed by $L^2$ stability analysis in \cite{LY18}:  the scheme is $L^2$ stable for $\beta_1 \geq 0$ and any $\beta_0\in \mathbb{R}$. Furthermore,  numerical convergence tests in \cite{LY18} for linear problems indicate that the following choices are sufficient for achieving optimal convergence,
\begin{subequations}\label{beta}
\begin{align}
\text{for (i)}   \qquad & \beta_1=\delta  ( k\geq 1);\\
\text{for (ii)} \qquad & |\beta_0|\geq C  \; ( k=1), \; \beta_0 =0 \;  ( k\geq 2),
\end{align}
\end{subequations}
where $k$ is the degree of underlying tensor polynomials, $\delta>0$ in (\ref{beta}a) can be a quite small number. For $P^1$ polynomials in one dimension, the optimal order of convergence is ensured even when $\beta_1=0$, as shown in \cite{LY18}. The choice of $C$ in (\ref{beta}b) is some constant.
For example, $C=3$ was used in one-dimensional tests in \cite[Example 5.5]{LY18}. For (iii),  optimal order of convergence  has been observed  in all related numerical tests in  \cite{LY18} and the present work.


\section{Time discretization}\label{sec3}
An appropriate time discretization should be adopted in order to preserve the energy dissipation law at each time step. One such discretization of (\ref{SemiDG}) studied in \cite{LY18} is  to obtain $(u_h^{n}, q_h^{n}) \in V_h \times V_h$ following the marching scheme,
\begin{subequations}\label{FPDGFullNon}
\begin{align}
   \left(  \frac{u_h^{n+1} - u_h^n}{\Delta t}, \phi \right) = & - A(q_h^{n+1/2},\phi)- \left( \frac{\Phi(u_h^{n+1})-\Phi(u^n)}{u_h^{n+1}-u_h^n},\phi\right), \\
    (q_h^{n}, \psi) = & A(u_h^{n},\psi),
\end{align}
\end{subequations}
for all $\phi, \ \psi \in V_h$, to approximate $u_h(\cdot, t_n)$, $q_h(\cdot, t_n)$, where $t_n=n\Delta t$ with $\Delta t$ being the time step.

This scheme is shown in \cite{LY18} to preserve  the energy dissipation law in the sense that
\bq\label{engdisJSC}
\mathcal{E}_h^{n+1} - \mathcal{E}_h^n = -\frac{\|u_h^{n+1}-u_h^{n}\|^2}{\Delta t} ,
\eq
where
$$
\mathcal{E}_h^n= \int_\Omega \Phi(u_h^n) +\frac{1}{2}|q_h^n|^2dx.
$$
However, implementation of (\ref{FPDGFullNon}) must involve some iteration, see a particular iteration for simulating the Swift--Hohenberg equation in \cite{LY18}.

Here following the idea of the IEQ method (cf.  \cite{Y16}), we propose both first and second order time discretization to the semi-discrete DG scheme (\ref{SemiDGN}) so that the schemes obtained are energy stable independent of time steps, and without resorting to any iteration method.
Because of (\ref{pa}),  we can choose a constant $B$ so that $\Phi(w)+B > 0, \ \forall w \in \mathbb{R}$,  and $U=\sqrt{\Phi(u_h)+B}$ is well-defined. The corresponding energy now reads as
\begin{align}\label{ee}
E(u_h, q_h, U)=\int_{\Omega} \left(\frac{1}{2}|q_h|^2+U^2 \right)dx + \frac{\alpha}{2 h}\int_{\Gamma^{\partial}} u_h^2ds =\mathcal E(u_h, q_h)+B|\Omega|.
\end{align}
With this notation we have $\Phi'(u_h)=H(u_h)U$ with
\begin{align}\label{hw}
 H(w)= \frac{\Phi'(w)}{\sqrt{\Phi(w)+B}}.
\end{align}
Instead of using the formula $U=\sqrt{\Phi(u_h)+B}$, we update $U$ by following its differentiation $U_t =  \frac{1}{2}H u_{ht}$.
More precisely, we consider  the following enlarged system:  find $(u_h(\cdot, t), q_h(\cdot, t)) \in V_h  \times V_h $  such that
\begin{subequations}\label{SemiDG++}
\begin{align}
U_{t} =&  \frac{1}{2}H(u_h) u_{ht},\\
(u_{ht}, \phi) +\alpha h^{-1} (u_h, \phi)_{\Gamma^{\partial}} = &- A(\phi, q_h) -\left(H(u_h)U,\phi \right),\\
(q_h, \psi) = & A(u_h,\psi),
\end{align}
\end{subequations}
for all $\phi, \psi \in V_h $. The initial data for the above scheme is chosen as
$$
u_h(x, 0)=\Pi u_0(x), \quad U(x, 0)=\sqrt{\Phi(u_0(x))+B},
$$
where $\Pi$ denotes the piecewise $L^2$ projection into $V_h$.

By taking $\phi=u_{ht}$ in  (\ref{SemiDG++}b) and $\psi =q_h$ in $(\ref{SemiDG++}c)_t$, which is a resulting equation from differentiation of
 (\ref{SemiDG++}c) in $t$,  upon further summation one can verify that
\bqs
\frac{d}{dt}E(u_h, q_h, U) = -\int_{\Omega}|u_{ht}|^2dx \leq 0,
\eqs
where $E(u_h, q_h, U)$
is the discrete energy for the enlarged system (\ref{SemiDG++}).

We are now ready to discretize (\ref{SemiDG++}) in time.
\subsection{First order fully discrete DG scheme}
Find  $(u^{n}_h, q_h^{n}) \in V_h  \times V_h $ and $U^n=U^n(x)$ such that
\begin{subequations}\label{FPDGFull1st+}
	\begin{align}
	U^n_h= & \Pi U^n, \\
	\frac{U^{n+1} -  U^n_h}{\Delta t} = & \frac{1}{2}H(u_h^n) \frac{u_h^{n+1} - u_h^n}{\Delta t},\\
	\left(  \frac{u_h^{n+1} - u_h^n}{\Delta t}, \phi \right) + \alpha h^{-1} (u_h^{n+1}, \phi)_{\Gamma^{\partial}} = & - A(\phi, q_h^{n+1})-\left( H(u_h^n)U^{n+1}, \phi \right), \\
(q_h^{n},\psi) = & A(u_h^{n}, \psi),
	\end{align}
\end{subequations}
for $\forall \phi, \psi \in V_h $, with initial data
$$
u_h^0=u_h(x, 0), \quad U^0= U(x, 0).
$$
Note that $U^n$ is not necessary in  $V_h$,  {but} $U_h^n\in V_h$.

Set
$$
E^n := E(u_h^n, q_h^n, U_h^n).
$$
For fully discrete DG scheme (\ref{FPDGFull1st+}), we have the following.
\begin{thm}\label{firstorder+}
	The fully discrete DG scheme (\ref{FPDGFull1st+}) admits a unique solution $(u_h^{n}, q_h^{n})$ for any $\Delta t>0$.
	 Moreover,
	\begin{align}\label{engdis1st+}
	E^{n+1} \leq  E^n - \frac{\| u_h^{n+1} - u_h^n\|^2}{\Delta t}-\frac{1}{2}\|q_h^{n+1} - q_h^n\|^2-\|U^{n+1} - U_h^n\|^2 -\frac{\alpha}{2h}\|u_h^{n+1}-u_h^n\|^2_{L^2(\Gamma^{\partial})},
	\end{align}
	independent of the size of $\Delta t$.
\end{thm}
\begin{proof}  We first show the existence and uniqueness of (\ref{FPDGFull1st+}) at each time step.  Substitution of (\ref{FPDGFull1st+}b) into (\ref{FPDGFull1st+}c)  with (\ref{FPDGFull1st+}d) gives  the following linear system
 \begin{subequations}\label{FPDGFullAlg+}
    \begin{align}
    \left( \left( \frac{1}{\Delta t} + \frac{H(u^{n}_h)^2}{2} \right)u^{n+1}_h, \phi\right)+\alpha h^{-1} (u_h^{n+1}, \phi)_{\Gamma^{\partial}}+   A(\phi, q^{n+1}_h)= &\left( f^n, \phi\right),\\
   A(u^{n+1}_h,\psi)-( q^{n+1}_h, \psi) = & 0,
    \end{align}
    \end{subequations}
where $f^n=u^{n}_h / \Delta t+1/2H(u^{n}_h)^2 u^{n}_h -H(u^{n}_h)U_h^n$ depends on solutions at $t=t_n$.  Taking $\phi=u_h^{n+1}$ and $\psi=q_h^{n+1}$ in  (\ref{FPDGFullAlg+}), upon subtraction and using $(f^n, \phi)\leq \frac{1}{2\Delta t}\|\phi\|^2+\frac{\Delta t}{2}\|f^n\|^2$ we obtain
$$
\|u_h^{n+1}\|^2 +2\Delta t \|q_h^{n+1}\|^2 + 2\Delta t \alpha h^{-1} \|u_h\|^2_{L^2(\Gamma^{\partial})} \leq \| \Delta t f^n\|^2.
$$
This stability estimate implies the uniqueness of the linear system (\ref{FPDGFullAlg+}), hence its existence  since for a linear system in finite dimensional space, existence is equivalent to its uniqueness.

We next prove (\ref{engdis1st+}). To this end, we define a notation $D_t u_h^n=\frac{u_h^{n+1}-u_h^{n}}{\Delta t}$, also for $q_h^n$. From (\ref{FPDGFull1st+}d), it follows
\begin{align}
(D_tq_h^n, \psi)=A(D_tu_h^n, \psi).
\end{align}
Taking $\psi=q_h^{n+1}$ and $\phi=D_t u_h^n$ in (\ref{FPDGFull1st+}c), when combined and using (\ref{FPDGFull1st+}b) we have
\begin{align*}
-\|D_tu_h^n\|^2  =& \alpha h^{-1} (u_h^{n+1}, D_tu_h^n)_{\Gamma^{\partial}}+(D_t q_h^{n}, q_h^{n+1}) +(H(u_h^n)U^{n+1}, D_tu_h^n)\\
 =&\frac{\alpha}{2h} \left(D_t \|u_h^n\|^2_{L^2(\Gamma^{\partial})} +\Delta t \|D_t u_h^n\|^2_{L^2(\Gamma^{\partial})}\right)  +  \frac{1}{2} D_t \|q_h^n\|^2 +\frac{\Delta t }{2} \|D_t q_h^n\|^2 \\
 & +\frac{2}{\Delta t} \left(U^{n+1}, U^{n+1}-U_h^n \right)\\
 = &\frac{\alpha}{2h} \left(D_t \|u_h^n\|^2_{L^2(\Gamma^{\partial})} +\Delta t \|D_t u_h^n\|^2_{L^2(\Gamma^{\partial})}\right)+ \frac{1}{2} D_t \|q_h^n\|^2 +\frac{\Delta t }{2} \|D_t q_h^n\|^2 \\
&+\frac{1}{\Delta t} \left(\|U^{n+1}\|^2-\|U_h^n\|^2 + \| U^{n+1}-U_h^n\|^2 \right).
\end{align*}
This is nothing but the following identity
\bq\label{engdis1stO}
\begin{aligned}
E(u_h^{n+1}, q_h^{n+1}, U^{n+1}) = & E(u_h^n, q_h^n, U_h^n) - \frac{\| u_h^{n+1} - u_h^n\|^2}{\Delta t}- \frac{\alpha}{2h}\| u_h^{n+1} - u_h^n\|^2_{L^2(\Gamma^{\partial})}\\
& -\frac{1}{2}\|q_h^{n+1} - q_h^n\|^2-\|U^{n+1} - U_h^n\|^2.
\end{aligned}
\eq
Implied by the fact that $\Pi$ is a contraction mapping in $L^2$,  we have
\bq\label{engdis1stO1}
E(u_h^{n+1}, q_h^{n+1}, U_h^{n+1}) \leq E(u_h^{n+1}, q_h^{n+1}, U^{n+1}),
\eq
hence (\ref{engdis1st+}) as desired.

\end{proof}

\subsection{Second order fully discrete DG scheme}  Here the time discretization
is done in a symmetric fashion around the point $t_{n+1/2}=(n+1/2)\Delta t$, which will produce a second order accurate method in time.
Denote by $v^{n+1/2}=(v^n+v^{n+1})/2$ for $v=u_h, q_h$, we find $(u^{n}_h, q_h^{n}) \in V_h  \times V_h $ such that for $\forall \phi, \psi
\in V_h $,
\begin{subequations}\label{FPDGFull+}
	\begin{align}
	U_h^n= & \Pi U^n,\\
  \frac{U^{n+1} - U_h^n}{\Delta t}  =&  \frac{1}{2} H(u^{n,*}_h) \frac{u_h^{n+1} - u_h^n}{\Delta t},\\
	\left(  \frac{u_h^{n+1} - u_h^n}{\Delta t}, \phi \right) +\alpha h^{-1} (u_h^{n+1/2}, \phi)_{\Gamma^{\partial}}
	= & - A(\phi, q_h^{n+1/2})-\frac{1}{2}\left(H(u^{n,*}_h)(U^{n+1}+U_h^n),\phi \right),\\
	(q_h^{n}, \psi) = & A(u_h^{n},\psi),
	\end{align}
\end{subequations}
where $u^{n,*}_h$ is obtained using $u_h^{n-1}$ and $u^n_h$ by
\begin{align}\label{u8}
u^{n, *}_h=& \frac{3}{2}u_h^n-\frac{1}{2}u_h^{n-1}.
\end{align}
Here instead of  $u_h^{n+1/2}$ we use $u^{n, *}_h$ to avoid the use of iteration steps in updating the numerical solution, while still maintaining second order accuracy in time. When $n=0$ in (\ref{u8}), we simply take  $u_h^{-1}=u_h^{0}$.

For the obtained discrete DG scheme (\ref{FPDGFull+}), we have
\begin{thm}\label{thm2nd}
The  fully discrete DG scheme (\ref{FPDGFull+}) admits a unique solution for any $\Delta t >0$.  Moreover, such scheme satisfies the following discrete energy dissipation law,
\bq\label{engdis}
E^{n+1} \leq E(u_h^{n+1}, q_h^{n+1}, U^{n+1}) =  E^n - \frac{\| u_h^{n+1} - u_h^n\|^2}{\Delta t},
\eq
independent of the size of $\Delta t$.
\end{thm}

\begin{proof}

We first prove (\ref{engdis}). We continue to use the notation $D_t v^n=\frac{v^{n+1}-v^{n}}{\Delta t}$. From (\ref{FPDGFull+}), it  follows
\bq
(D_t q_h^n, \psi) = A(D_t u_h^n, \psi).
\eq
Taking $\psi=q_h^{n+1/2}$ and $\phi=D_t u_h^n$ in (\ref{FPDGFull+}c), when combined with (\ref{FPDGFull+}b) we have
\begin{align*}
-\|D_tu_h^n\|^2 & = \alpha h^{-1} (u_h^{n+1/2}, D_tu_h^n)_{\Gamma^{\partial}}+ (D_t q_h^{n}, q_h^{n+1/2}) +\frac{1}{2} (H(u_h^{n,*})(U^{n+1}+U_h^n), D_tu_h^n)\\
& =\frac{\alpha}{2h}D_t \|u_h^n\|^2_{L^2(\Gamma^{\partial})}+ \frac{1}{2} D_t \|q_h^n\|^2 +\frac{1}{\Delta t} \left(\|U^{n+1}\|^2-\|U_h^n\|^2\right).
\end{align*}
Multiplying by $\Delta t$ on both sides of this equality, we have
\begin{equation}\label{engdis2nd}
E(u_h^{n+1}, q_h^{n+1}, U^{n+1}) =  E(u_h^n,q_h^n, U_h^n) - \frac{\| u_h^{n+1} - u_h^n\|^2}{\Delta t},
\end{equation}
which combining with (\ref{engdis1stO1}) leads to (\ref{engdis}).

For the uniqueness, we let  $(\tilde u, \tilde q, \tilde U)$ be the difference of two possible solutions at $t=t_{n+1}$, then a similar analysis to the above yields
$$
E(\tilde u, \tilde q, \tilde U)+\frac{\|\tilde u \|^2}{\Delta t} =0,
$$
hence we must  have $(\tilde u, \tilde q, \tilde U)=(0, 0, 0)$, leading to the uniqueness of the full system (\ref{FPDGFull+}).
\end{proof}

\subsection{Algorithm}
The detail related to the scheme implementation is summarized in the following algorithm (for second order scheme (\ref{FPDGFull+}) only, that for first order scheme (\ref{FPDGFull1st+}) is simpler).
\begin{itemize}
  \item Step 1 (Initialization), from the given initial data $u_0(x)$
  \begin{enumerate}
    \item generate $u_h^0 =\Pi u_0(x) \in V_h $, set $u_h^{-1}=u_h^0$,
    \item solve for $q_h^0$ from (\ref{FPDGFull+}d) based on $u_h^0$, and
    \item generate $U^0= \sqrt{\Phi(u_0(x)) +B}$, where $B$ is a priori chosen so that $\inf \Phi(w)+B>0$.
  \end{enumerate}
  \item Step 2 (Evolution)
  \begin{enumerate}
  \item Project $U^n$ back into $V_h$, $U^n_h=\Pi U^n$;
    \item Solve the following linear system
    \begin{subequations}\label{FPDGFullAlg}
    \begin{align}
    \left( \left( \frac{1}{\Delta t} + \frac{H(u^{n, *}_h)^2}{4} \right)u^{n+1}_h, \phi\right)+  \frac{1}{2} A(\phi, q^{n+1}_h)+ \frac{\alpha}{2h} (u_h^{n+1}, \phi)_{\Gamma^{\partial}} = & RHS,\\
    \frac{1}{2} A(u^{n+1}_h,\psi)-\frac{1}{2} ( q^{n+1}_h, \psi) = & 0,
    \end{align}
    \end{subequations}
    where $ u^{n, *}_h=\frac{3}{2}u^n_h-\frac{1}{2}u_h^{n-1}$, and
    $$
    RHS=\left( f^n, \phi\right)- \frac{1}{2}A(\phi, q^{n}_h) - \frac{\alpha}{2h} (u_h^{n}, \phi)_{\Gamma^{\partial}},
    $$
    with $ f^n=u^{n}_h / \Delta t+1/4H(u^{n, *}_h)^2 u^{n}_h-H(u^{n, *}_h)U_h^n$.
      \item Update $U^{n+1}$ using (\ref{FPDGFull+}), then return to (1) in Step 2.
  \end{enumerate}
\end{itemize}
{Note that (\ref{FPDGFullAlg}) is a linear system with sparse coefficient matrix which is changing at each time step, we solve it by the open source deal.II  finite element library as documented in \cite{BHK07}, using an incomplete LU factorization as a preconditioner and preconditioned
flexible GMRES as a solver.}

\begin{rem} {Recently the SAV method has been introduced in \cite{SY18}  with certain advantages over the $IEQ$.  The basic idea when applied to the present setting is to introduce a scalar auxiliary variable $r=\sqrt{\int_{\Omega} \Phi(u_h)dx +B }$, and update $r$ by $r_t=\frac{1}{2r}\int_{\Omega} \Phi'(u_h) u_{ht}dx$. A replacement of $(\Phi'(u_h^{n+1}), \phi)$ by $(\Phi'(u_h^{n}), \phi)\frac{r^{n+1}}{r^n}$ yields a  linearized scheme which can be shown unconditional energy stable. It appears more involved to solve the resulting system efficiently within the DG framework.
}
\end{rem}

\subsection{Fully discrete DG scheme for non-homogeneous boundary conditions}  For non-homogeneous boundary conditions (i)-(iii) given in (\ref{3bd}), the fully discrete DG schemes for (\ref{SemiDGN}) need to be modified.

For the first order fully discrete DG scheme (\ref{FPDGFull1st+}),  equations (\ref{FPDGFull1st+}c) and (\ref{FPDGFull1st+}d) need to be modified as
\begin{equation*}
\begin{aligned}
\left(  \frac{u_h^{n+1} - u_h^n}{\Delta t}, \phi \right) + \alpha h^{-1} (u_h^{n+1}, \phi)_{\Gamma^{\partial}} = & - A(\phi, q_h^{n+1})-\left( H(u_h^n)U^{n+1}, \phi \right)+ L_1(t^{n+1};\phi), \\
(q_h^{n},\psi) = & A(u_h^{n}, \psi)+ L_2(t^{n};\psi).
\end{aligned}
\end{equation*}

For the second order fully discrete DG scheme (\ref{FPDGFull+}), equations  (\ref{FPDGFull+}c)  and (\ref{FPDGFull+}d) need to be modified as
\begin{equation*}
\begin{aligned}
	\left(  \frac{u_h^{n+1} - u_h^n}{\Delta t}, \phi \right) +\alpha h^{-1} (u_h^{n+1/2}, \phi)_{\Gamma^{\partial}}
	= & - A(\phi, q_h^{n+1/2})-\frac{1}{2}\left(H(u^{n,*}_h)(U^{n+1}+U_h^n),\phi \right)\\
   & + \frac{1}{2}L_1(t^{n+1};\phi) + \frac{1}{2}L_1(t^{n};\phi), \\
	(q_h^{n}, \psi) = & A(u_h^{n},\psi)+ L_2(t^{n};\psi).
\end{aligned}
\end{equation*}
{It is known that for non-homogeneous boundary conditions given in (\ref{3bd}), the energy dissipation law  (\ref{engdisO}) needs to be replaced by
\bq\label{engdisO+}
\frac{d}{dt}\mathcal{E}(u) =  -\int_{\Omega}|u_t|^2dx +J,
\eq
where the boundary contribution $J=\int_{\partial \Omega} (u_t \partial_\nu q-\partial_\nu u_t q)ds$ with $q=-(\Delta +a/2)$ depends on
the available boundary data and the involved solution traces.  For the above two schemes, energy variation in time can be  derived in entirely similar manner to that leading to (\ref{engdis1st+}) and (\ref{engdis}), respectively, with attention necessary only on boundary contributions.}

\section{Numerical examples}
In this section we numerically test the orders of convergence in both spatial and temporal discretization, and the unconditional energy stability; further apply  scheme  (\ref{FPDGFull+})  to recover some known patterns governed by the 2D Swift--Hohenberg equation.  The errors between the numerical solution $u_h^n(x, y)$ and the exact solution or a reference solution $u(t^n,x,y)$ are evaluated in the following manner. 
{ The 2D $L^\infty$ error is given by
 $$
e_h^n = \max_{i}\max_{0\leq l \leq G}\max_{0\leq s \leq G} |u_h^n(\hat{x}^i_{l},\hat{y}^i_{s})-u(t^n,\hat{x}^i_{l},\hat{y}^i_{s})|,
$$
and the $L^2$ error is given by
$$
e_h^n = \left(\sum_{i} \frac{h^i_xh^i_y}{4} \sum_{l=1}^{G}\sum_{s=1}^{G} \omega_{l,s} |u_h^n(\hat{x}^i_{l},\hat{y}^i_{s})-u(t^n,\hat{x}^i_{l},\hat{y}^i_{s})|^2\right)^{1/2},
$$
where $\omega_{l,s}>0$ are the weights, and $(\hat{x}^i_{l},\hat{y}^i_{s})$ are the corresponding quadrature points for $G \geq k+1$.}
The experimental orders of convergence (EOC) at $T=n\Delta t=2n (\Delta t /2)$ in terms of $h$ and $\Delta t$ are then determined respectively by
$$
\text{EOC}=\log_2 \left( \frac{e_h^n}{e_{h/2}^n}\right),  \quad \text{EOC}=\log_2 \left( \frac{e_h^n}{e_{h}^{2n}}\right).
$$
{ Different choices for $B$, as numerically verified in most cases, can work equally well, so we take $B=1$ for all examples except in Example \ref{varmethods}.}  In our numerical examples we output $E(u_h^n, q_h^n, U_h^n)-B|\Omega|$ instead of $E(u_h^n, q_h^n, U_h^n)$  to  better observe  the evolution of the original free energy $\mathcal{E}_h^n$.

Note that our numerical scheme is established for the model equation (\ref{fourthPDE}), which includes the Swift--Hohenberg equation (\ref{SH}) as a special case with $a=2$ and
$$
 \Psi(u)=\frac{1-\epsilon}{2}u^2 -\frac{g}{3}u^3 +\frac{u^4}{4},
$$
modulo an additive constant. For any $g$,  such $\Psi$ satisfies (\ref{pa}), which is necessary for the use of the IEQ approach.  In  the following numerical examples we focus mainly on  the Swift--Hohenberg equation with different choices of $\epsilon$ and/or $g$.

\begin{example}\label{Ex1dAccS} (Spatial Accuracy Test)


Consider the Swift--Hohenberg equation (\ref{SH}) by adding a source term
$f(x,y, t)=- \varepsilon v -gv^2+ v^3$ with $v=e^{-t/4}\sin(x/2)\sin(y/2)$
for some parameters $\varepsilon, g$,  and the initial data
\bq\label{initex1}
u_0(x,y) = \sin(x/2)\sin(y/2), \quad (x, y) \in \Omega.
\eq
Its exact solution is given by
\bq\label{uexact}
u(x,y,t) = e^{-t/4}\sin(x/2)\sin(y/2), \quad (x, y) \in \Omega.
\eq
This example is to test the spatial accuracy on 2D rectangular meshes, subject to different types of boundary conditions,
we use the second-order fully discrete DG scheme (\ref{FPDGFull+}) with
$$
\frac{1}{2}\left(f(\cdot, t^{n+1}, \phi)+f(\cdot, t^{n}, \phi)\right),
$$
added to the right hand side of (\ref{FPDGFull+}c)
using polynomials of degree $k$ with $k =1,\ 2,\ 3$.

\noindent\textbf{Test case 1.} (Periodic boundary conditions) For parameters $\varepsilon=0.025, g=0$ and domain $\Omega=[-2\pi, 2\pi]^2$ with periodic boundary conditions.  Both errors and orders of convergence at $T=0.1$ are reported in Table \ref{tab2dacc}. These results confirm the $(k+1)$th orders of accuracy in $L^2, L^\infty$ norms.

\begin{table}[!htbp]\tabcolsep0.03in
\caption{$L^2, L^\infty$ errors and EOC at $T = 0.1$ with mesh $N\times N$.}
\begin{tabular}[c]{||c|c|c|c|c|c|c|c|c|c||}
\hline
\multirow{2}{*}{$k$} & \multirow{2}{*}{$\Delta t$}&   \multirow{2}{*}{ } & N=8 & \multicolumn{2}{|c|}{N=16} & \multicolumn{2}{|c|}{N=32} & \multicolumn{2}{|c||}{N=64}  \\
\cline{4-10}
& & & error & error & order & error & order & error & order\\
\hline
\multirow{2}{*}{1}  & \multirow{2}{*}{1e-3} & $\|u-u_h\|_{L^2}$ &  3.96917e-01 & 9.53330e-02 & 2.06 & 2.34412e-02 & 2.02 & 5.86903e-03 & 2.00  \\
\cline{3-10}
 & & $\|u-u_h\|_{L^\infty}$  & 1.46432e-01 & 3.75773e-02 & 1.96 & 9.40110e-03 & 2.00 & 2.35038e-03 & 2.00  \\
\hline
\hline
\multirow{2}{*}{2}  & \multirow{2}{*}{1e-4} & $\|u-u_h\|_{L^2}$ & 1.00063e-01 & 1.48191e-02 & 2.76 & 1.98345e-03 & 2.90 & 2.60819e-04 & 2.93  \\
\cline{3-10}
 & & $\|u-u_h\|_{L^\infty}$  & 2.57951e-02 & 3.16978e-03 & 3.02 & 4.30633e-04 & 2.88 & 5.61561e-05 & 2.94  \\
 \hline
\hline
\multirow{2}{*}{3} & \multirow{2}{*}{1e-5}  & $\|u-u_h\|_{L^2}$ & 1.34590e-02 & 1.10668e-03 & 3.60 & 7.55223e-05 & 3.87 & 4.83308e-06 & 3.97  \\
\cline{3-10}
 & & $\|u-u_h\|_{L^\infty}$   & 4.07154e-03 & 3.60524e-04 & 3.50 & 2.38081e-05 & 3.92 & 1.51432e-06 & 3.97  \\
\hline
\end{tabular}\label{tab2dacc}
\end{table}

\noindent\textbf{Test case 2.} For parameters $\varepsilon=0.025, g=0$ and domain $\Omega=[0, 2\pi]^2$ with boundary condition $u = \Delta u = 0, \ (x,y) \in \partial \Omega$,  we use
 scheme (\ref{FPDGFull+}) with  $\alpha=0$ and $\beta_0=0$ in (\ref{bd}c). Both  errors and orders of convergence at $T=0.1$ are reported in Table \ref{tab2daccDir2}. These results also show that $(k+1)$th orders of accuracy in $L^2, L^\infty$ norms are obtained.

\begin{table}[!htbp]\tabcolsep0.03in
\caption{$L^2, L^\infty$ errors and EOC at $T = 0.1$ with mesh $N\times N$.}
\begin{tabular}[c]{||c|c|c|c|c|c|c|c|c|c||}
\hline
\multirow{2}{*}{$k$} & \multirow{2}{*}{$\Delta t$}&   \multirow{2}{*}{ } & N=8 & \multicolumn{2}{|c|}{N=16} & \multicolumn{2}{|c|}{N=32} & \multicolumn{2}{|c||}{N=64}  \\
\cline{4-10}
& & & error & error & order & error & order & error & order\\
\hline
\multirow{2}{*}{1}  & \multirow{2}{*}{1e-3} & $\|u-u_h\|_{L^2}$ &  4.76650e-02 & 1.17160e-02 & 2.02 & 2.91618e-03 & 2.01 & 7.28242e-04 & 2.00  \\
\cline{3-10}
 & & $\|u-u_h\|_{L^\infty}$  & 3.75725e-02 & 9.39988e-03 & 2.00 & 2.35007e-03 & 2.00 & 5.87520e-04 & 2.00  \\
\hline
\hline
\multirow{2}{*}{2}  & \multirow{2}{*}{1e-4} & $\|u-u_h\|_{L^2}$ & 7.40928e-03 & 9.91089e-04 & 2.90 & 1.26183e-04 & 2.97 & 1.58469e-05 & 2.99  \\
\cline{3-10}
 & & $\|u-u_h\|_{L^\infty}$  & 3.22366e-03 & 4.35251e-04 & 2.89 & 5.55145e-05 & 2.97 & 6.97483e-06 & 2.99  \\
 \hline
\hline
\multirow{2}{*}{3} & \multirow{2}{*}{5e-5}  & $\|u-u_h\|_{L^2}$ & 5.53341e-04 & 3.77612e-05 & 3.87 & 2.41654e-06 & 3.97 & 1.51952e-07 & 3.99  \\
\cline{3-10}
 & & $\|u-u_h\|_{L^\infty}$   & 3.60523e-04 & 2.38081e-05 & 3.92 & 1.51433e-06 & 3.97 & 9.51458e-08 & 3.99  \\
\hline
\end{tabular}\label{tab2daccDir2}
\end{table}

\noindent\textbf{Test case 3.} For parameters $\varepsilon=0.025, g=0.05$ and domain $\Omega=[-\pi, \pi]^2$ with  boundary condition $\partial_\nu u = \partial_\nu \Delta u = 0, \ (x,y) \in \partial \Omega$.  Both errors and orders of convergence at $T=0.1$ are reported in Table \ref{tab2daccNeu}. These results also show that $(k+1)$th orders of accuracy in both $L^2$ and $L^\infty$ norms are obtained.

\begin{table}[!htbp]\tabcolsep0.03in
\caption{$L^2, L^\infty$ errors and EOC at $T = 0.1$ with mesh $N\times N$.}
\begin{tabular}[c]{||c|c|c|c|c|c|c|c|c|c||}
\hline
\multirow{2}{*}{$k$} & \multirow{2}{*}{$\Delta t$}&   \multirow{2}{*}{ } & N=8 & \multicolumn{2}{|c|}{N=16} & \multicolumn{2}{|c|}{N=32} & \multicolumn{2}{|c||}{N=64}  \\
\cline{4-10}
& & & error & error & order & error & order & error & order\\
\hline
\multirow{2}{*}{1}  & \multirow{2}{*}{1e-3} & $\|u-u_h\|_{L^2}$ &  4.76652e-02 & 1.17160e-02 & 2.02 & 2.91618e-03 & 2.01 & 7.28242e-04 & 2.00  \\
\cline{3-10}
 & & $\|u-u_h\|_{L^\infty}$  & 3.75721e-02 & 9.39988e-03 & 2.00 & 2.35007e-03 & 2.00 & 5.87520e-04 & 2.00  \\
\hline
\hline
\multirow{2}{*}{2}  & \multirow{2}{*}{1e-4} & $\|u-u_h\|_{L^2}$ & 7.40926e-03 & 9.91089e-04 & 2.90 & 1.26183e-04 & 2.97 & 1.58469e-05 & 2.99  \\
\cline{3-10}
 & & $\|u-u_h\|_{L^\infty}$  & 3.22365e-03 & 4.35251e-04 & 2.89 & 5.55145e-05 & 2.97 & 6.97483e-06 & 2.99  \\
 \hline
\hline
\multirow{2}{*}{3} & \multirow{2}{*}{5e-5}  & $\|u-u_h\|_{L^2}$ & 5.53341e-04 & 3.77611e-05 & 3.87 & 2.41654e-06 & 3.97 & 1.51951e-07 & 3.99  \\
\cline{3-10}
 & & $\|u-u_h\|_{L^\infty}$   & 3.60523e-04 & 2.38081e-05 & 3.92 & 1.51405e-06 & 3.97 & 9.53835e-08 & 3.99  \\
\hline
\end{tabular}\label{tab2daccNeu}
\end{table}

\end{example}

\begin{example} In this example, we consider the problem with both a source and non-homogeneous boundary conditions of type (i) in
(\ref{3bd}):
\begin{align*}
    & u_t =  -(\Delta +1)^2 u +0.025 u-u^3 + f(x,y, t) \quad (x,y,t) \in [0, 2\pi]\times [0, 2\pi]\times (0,T],\\
    & u(x,y,0) =  \sin(x/2)\sin(y/2),\\
    & u(0,y,t)=u(2\pi,y,t)=u(x,0,t)=u(x,2\pi,t)=0, \\
    & \partial_x u (0,y,t) = 1/2 e^{-t/4}\sin(y/2), \ \partial_x u(2\pi,y,t)=-1/2 e^{-t/4}\sin(y/2), \\
    & \partial_y  u(x,0,t)=1/2 e^{-t/4}\sin(x/2), \  \partial_y u(x,2\pi,t)=-1/2 e^{-t/4}\sin(x/2),
\end{align*}
where $f(x,y, t)=- 0.025 v + v^3$ with $v=e^{-t/4}\sin(x/2)\sin(y/2)$.  Its exact solution is given by (\ref{uexact}).
We test the second order fully discrete DG scheme (\ref{FPDGFull+}) with
$$
\frac{1}{2}(L_1(t^{n+1}; \phi)+L_1(t^n; \phi)) +\frac{1}{2}(f(\cdot, t^{n+1}, \phi)+f(\cdot, t^{n}, \phi))
$$
added to (\ref{FPDGFull+}c) and $L_2(t^n; \psi)$ added to (\ref{FPDGFull+}d), based on $P^k$ polynomials with $k=1,2, 3$.
The flux parameter $\beta_1=1$.  Both the errors and orders of convergence at $T=0.1$ are reported in Table \ref{tab2daccDir1}. These results show that $(k+1)$th orders of accuracy in both $L^2$ and  $L^\infty$ are obtained.
\begin{table}[!htbp]\tabcolsep0.03in
\caption{$L^2, L^\infty$ errors and EOC at $T = 0.1$ with mesh $N\times N$.}
\begin{tabular}[c]{||c|c|c|c|c|c|c|c|c|c||}
\hline
\multirow{2}{*}{$k$} & \multirow{2}{*}{$\Delta t$}&   \multirow{2}{*}{ } & N=8 & \multicolumn{2}{|c|}{N=16} & \multicolumn{2}{|c|}{N=32} & \multicolumn{2}{|c||}{N=64}  \\
\cline{4-10}
& & & error & error & order & error & order & error & order\\
\hline
\multirow{2}{*}{1}  & \multirow{2}{*}{1e-3} & $\|u-u_h\|_{L^2}$ &  5.12416e-02 & 1.26151e-02 & 2.02 &3.33581e-03 & 1.92 & 9.37490e-04 & 1.83  \\
\cline{3-10}
 & & $\|u-u_h\|_{L^\infty}$  & 4.53223e-02 & 1.29022e-02 & 1.81 & 3.56895e-03 & 1.85 & 1.07682e-03 & 1.73  \\
\hline
\hline
\multirow{2}{*}{2}  & \multirow{2}{*}{1e-4} & $\|u-u_h\|_{L^2}$ & 6.90060e-03 & 1.10206e-03 & 2.65 & 1.34465e-04 & 3.03 & 1.63762e-05 & 3.04  \\
\cline{3-10}
 & & $\|u-u_h\|_{L^\infty}$  & 2.82321e-03 & 5.84377e-04 & 2.27 & 7.6956e-05 & 2.92 & 9.66781e-06 & 2.99  \\
 \hline
\hline
\multirow{2}{*}{3} & \multirow{2}{*}{1e-5}  & $\|u-u_h\|_{L^2}$ & 5.98414e-04 & 4.09284e-05 & 3.87 & 2.52723e-06 & 4.02 & 1.59071e-07 & 3.99  \\
\cline{3-10}
 & & $\|u-u_h\|_{L^\infty}$   & 5.14633e-04 & 5.04236e-05 & 3.35 & 3.18953e-06 & 3.98 & 1.91613e-07 & 4.06  \\
\hline
\end{tabular}\label{tab2daccDir1}
\end{table}

\end{example}

\begin{example}\label{Ex1dAccT} (Temporal Accuracy Test)  Consider the Swift--Hohenberg equation (\ref{SH}) on the domain $\Omega= [-2\pi, 2\pi]^2$ with the parameters $\varepsilon=0.025$ and $g=0$, the initial data
\bq\label{initex12}
u_0(x,y) = \sin(x/4)\sin(y/4).
\eq
and generalized Neumann boundary conditions $\partial_\nu u = \partial_\nu \Delta u = 0, \ (x,y) \in \partial \Omega$.

We compute a reference solution at $T=2$ using DG schemes (\ref{FPDGFull1st+}) and (\ref{FPDGFull+}) based on $P^2$ polynomials with time step $\Delta t=2^{-8}$ and appropriate meshes.  Numerical solutions are produced using larger time steps $\Delta t=2^{-m}$ with $3\leq m\leq 6$.  The $L^2, L^\infty$ errors and orders of convergence are shown in Table \ref{timeacc}, and these results confirm that DG schemes (\ref{FPDGFull1st+}) and (\ref{FPDGFull+}) are first order and second order in time, respectively.

\begin{table}[!htbp]\tabcolsep0.03in
\caption{$L^2, L^\infty$ errors and EOC at $T = 2$ with time step $\Delta t$.}
\begin{tabular}[c]{||c|c|c|c|c|c|c|c|c|c||}
\hline
\multirow{2}{*}{Scheme} & \multirow{2}{*}{Mesh}&   \multirow{2}{*}{ } & $\Delta t=2^{-3}$ & \multicolumn{2}{|c|}{$\Delta t=2^{-4}$} & \multicolumn{2}{|c|}{$\Delta t=2^{-5}$} & \multicolumn{2}{|c||}{$\Delta t=2^{-6}$}  \\
\cline{4-10}
& & & error & error & order & error & order & error & order\\
\hline
\multirow{2}{*}{(\ref{FPDGFull1st+})}  & \multirow{2}{*}{$32\times32$} & $\|u-u_h\|_{L^2}$ &  8.19277e-02 & 4.11370e-02 & 0.99 & 1.96177e-02 & 1.07 & 8.50327e-03 & 1.21  \\
\cline{3-10}
 & & $\|u-u_h\|_{L^\infty}$  & 1.07659e-02 & 5.43477e-03 & 0.99 & 2.59422e-03 & 1.07 & 1.12483e-03 & 1.21  \\
\hline
\hline
\multirow{2}{*}{(\ref{FPDGFull+})}  & \multirow{2}{*}{$64\times64$} & $\|u-u_h\|_{L^2}$ & 7.31631e-03 & 1.40500e-03 & 2.38 & 3.09235e-04 & 2.18 & 6.97759e-05 & 2.15  \\
\cline{3-10}
 & & $\|u-u_h\|_{L^\infty}$  & 1.74374e-03 & 2.64806e-04 & 2.72 & 5.34755e-05 & 2.31 & 1.17938e-05 & 2.18  \\
 \hline
\end{tabular}\label{timeacc}
\end{table}

\end{example}

\begin{example}\label{Ex2dPatt} (2D energy evolution) Consider the Swift--Hohenberg equation (\ref{SH}) on rectangular domain $\Omega=[0,40]^2$ with parameters $\epsilon=2$, $g=0$, initial data
\begin{equation}\label{Exinit}
u(x,y,0)=
\left \{
\begin{array}{ll}
    1,&   x_1<x<x_2,\\
    -1,&  \text{otherwise},
\end{array}
\right.
\end{equation}
where $x_1=\sin\left(\frac{2\pi}{10}y\right)+15$ and $x_2=\cos\left(\frac{2\pi}{10}y\right)+25$ form a curvy vertical strip, and the boundary conditions
$\partial_\nu u = \partial_\nu \Delta u = 0, \ (x,y) \in \partial \Omega$. This example is taken from \cite{GN12}, using the equations of the curvy vertical strip described therein.
We solve this problem by scheme (\ref{FPDGFull+}) based on $P^2$ polynomials on $64 \times 64$ meshes.  The energy evolution in time with $t\in[0,10]$ for varying time steps are shown in Figure \ref{2du12eng}, from which we see that scheme (\ref{FPDGFull+}) is always energy dissipating for any $\Delta t$ as tested, however the size of $\Delta t$ appears to affect the decay rate of the energy.  These numerical results suggest that time step should be chosen with case. One possibility is to set up an energy threshold in such a way that if the energy is about such threshold, $\Delta t$ should be small, and after energy falls below the threshold, one can simply adjust to a larger time step.

Furthermore, the numerical solutions with $\Delta t =0.001$ are shown in Figure \ref{PatBifur}, which reveals a series of evolved patterns in time.  
The  energy evolution over a larger time interval is also given in Figure \ref{BifEng}, which again shows the energy dissipation property of numerical solutions.

 \begin{figure}
 \centering
 \subfigure[]{\includegraphics[width=0.325\textwidth]{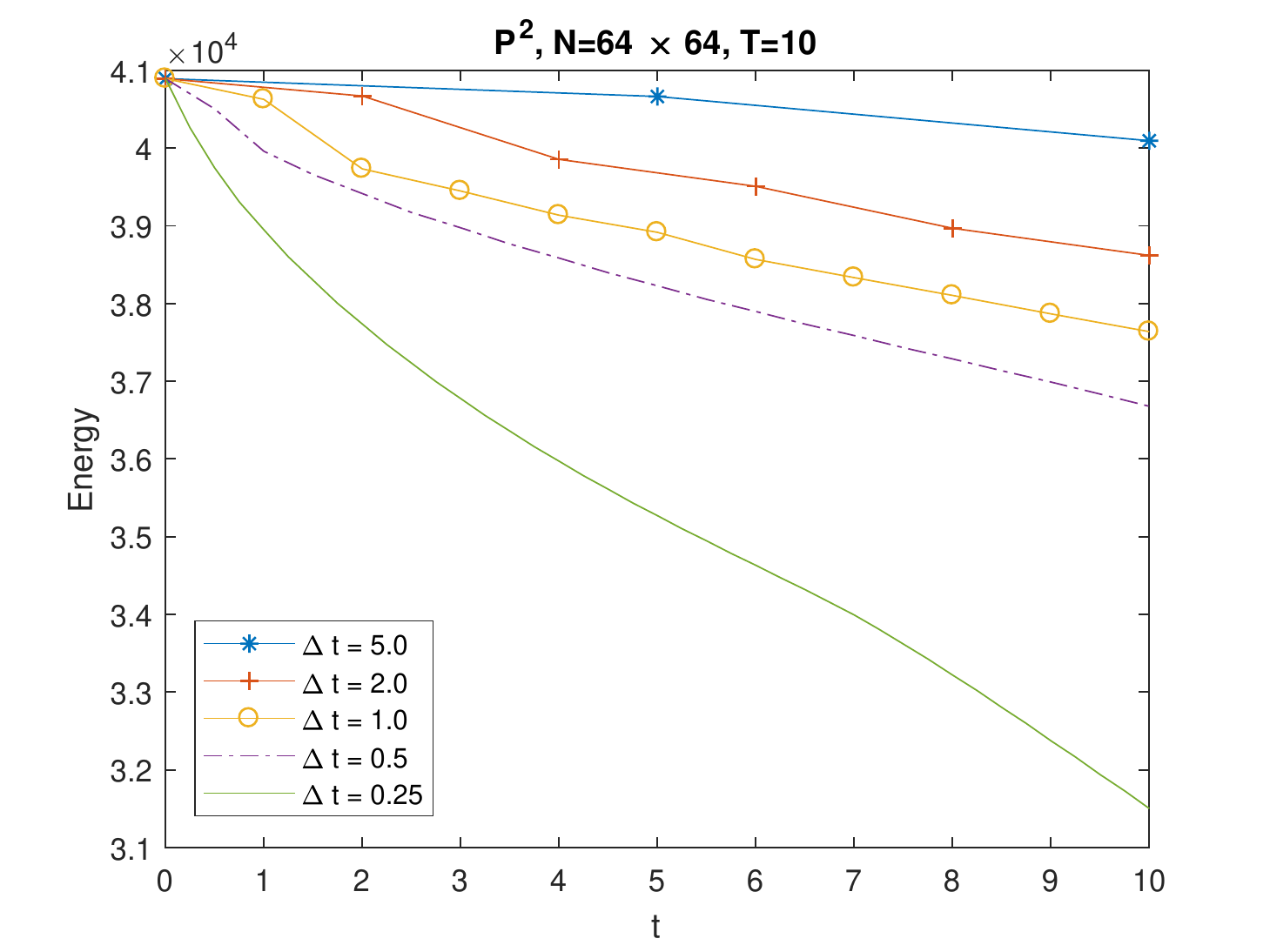}}
 \subfigure[]{\includegraphics[width=0.325\textwidth]{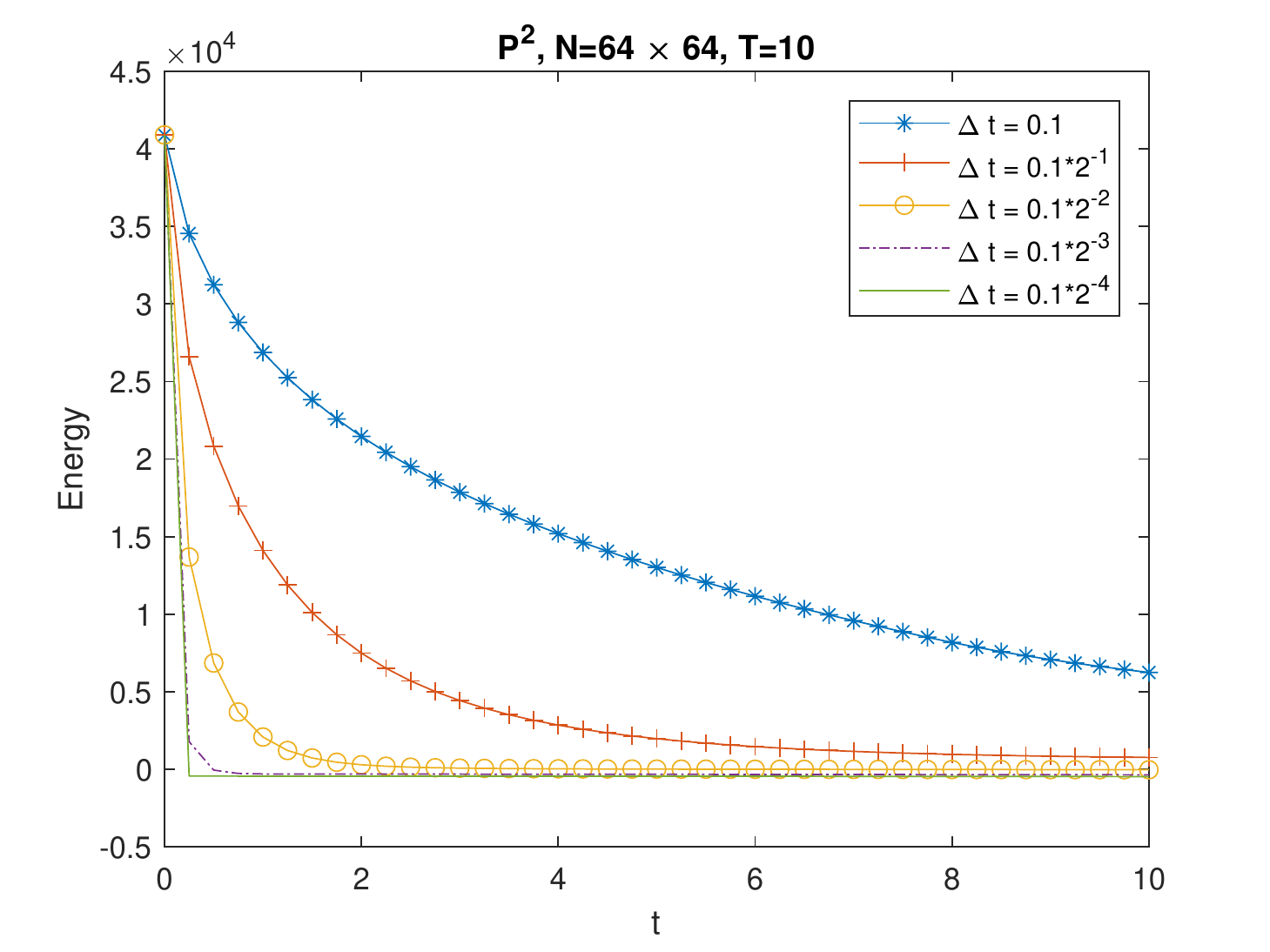}}
 \subfigure[]{\includegraphics[width=0.325\textwidth]{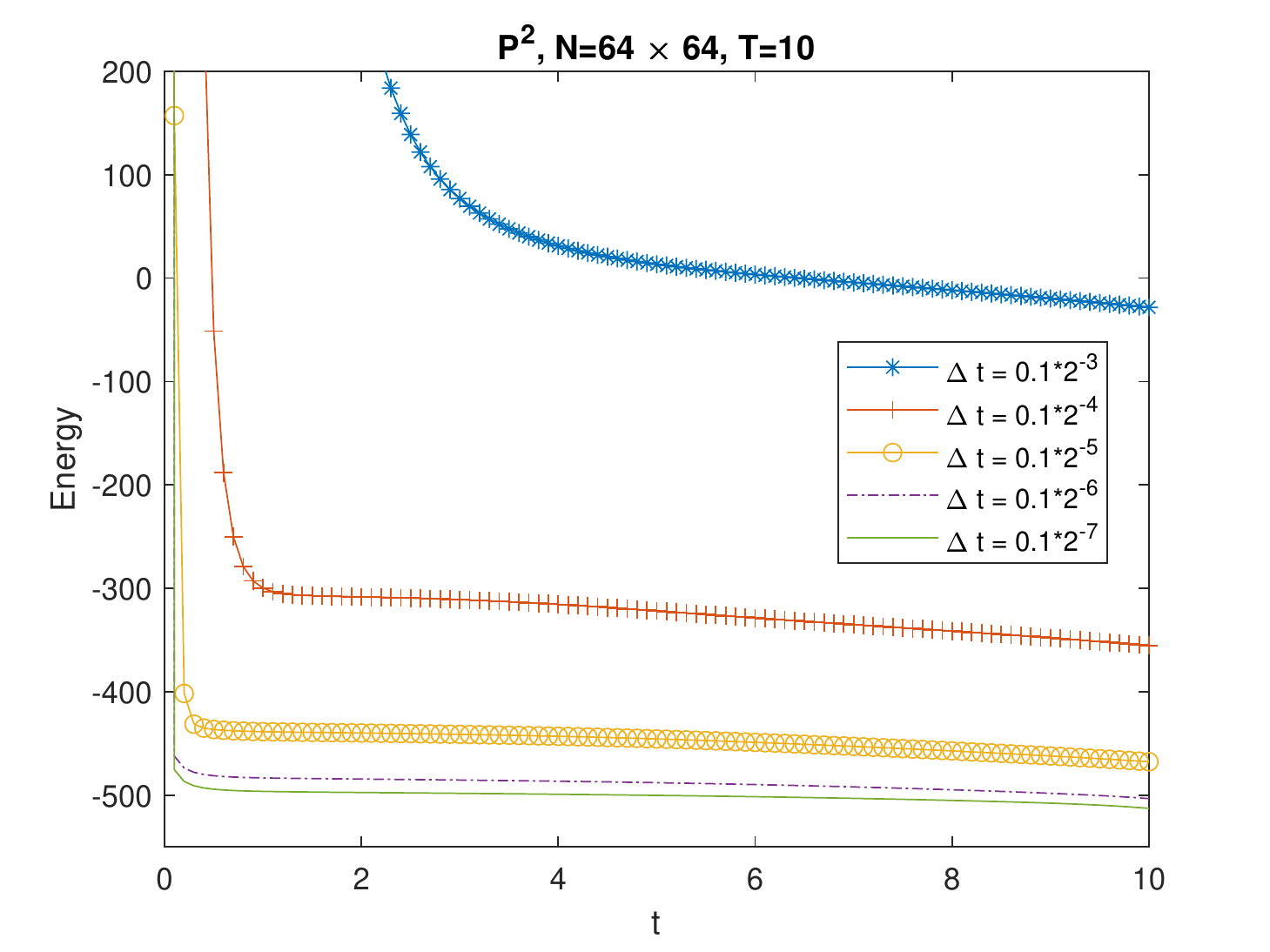}}
 \caption{ Energy evolution for several time steps using the DG scheme (\ref{FPDGFull+}), (a) normal view, (b) normal view, (c) zoomed view.
 } \label{2du12eng}
 \end{figure}

 \begin{figure}
 \centering
 \subfigure{\includegraphics[width=0.325\textwidth]{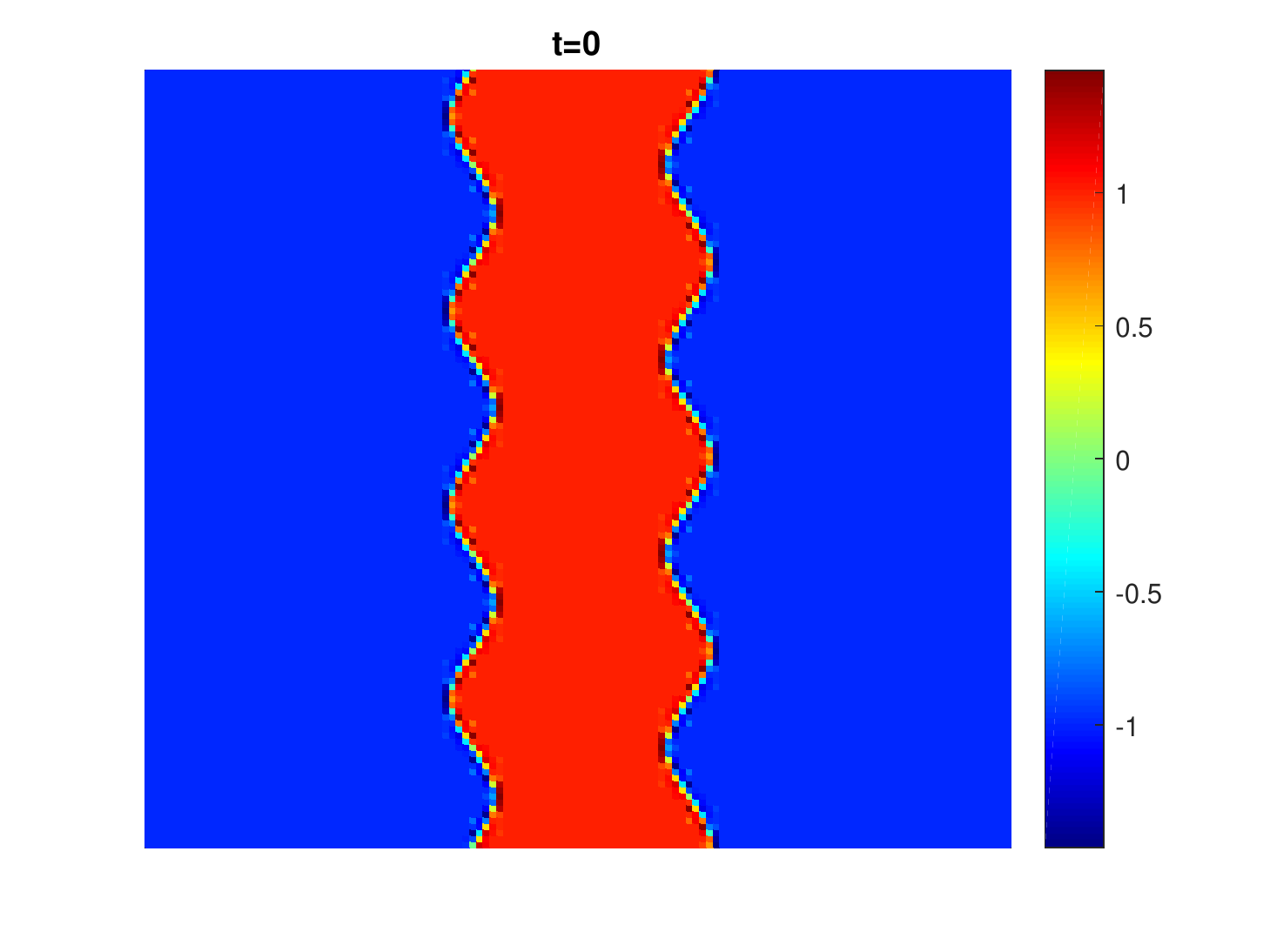}}
 \subfigure{\includegraphics[width=0.325\textwidth]{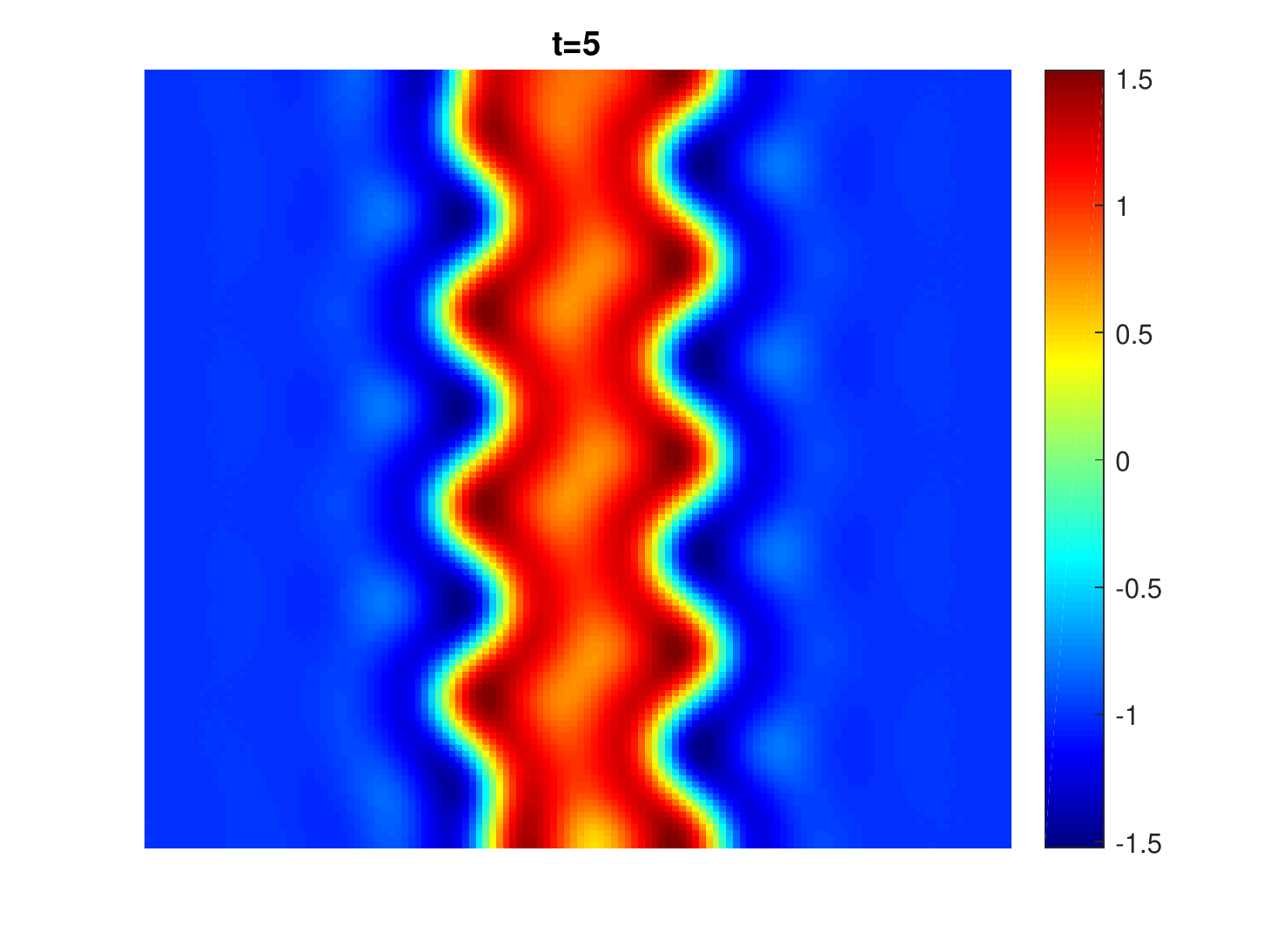}}
 \subfigure{\includegraphics[width=0.325\textwidth]{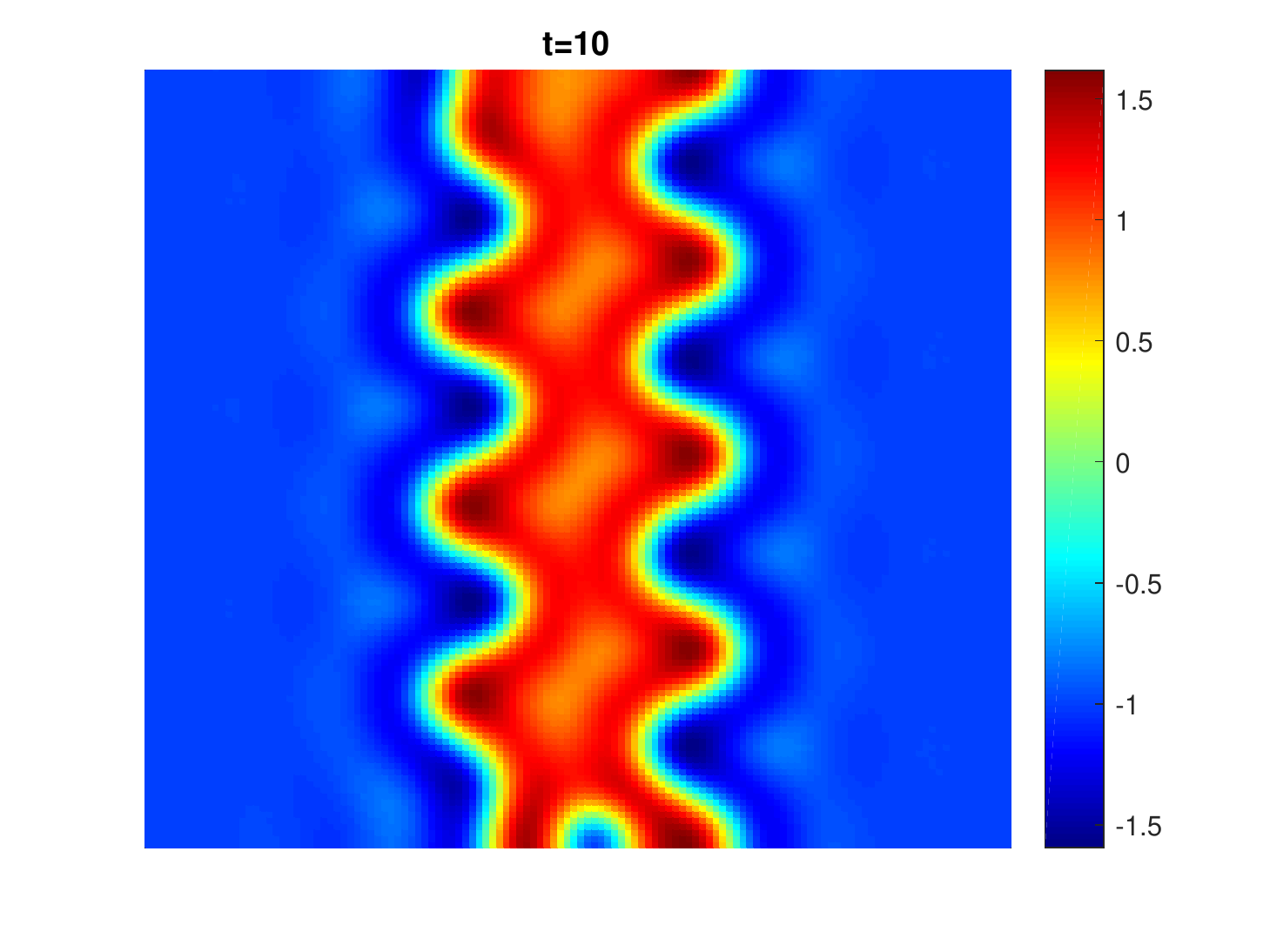}}
 \subfigure{\includegraphics[width=0.325\textwidth]{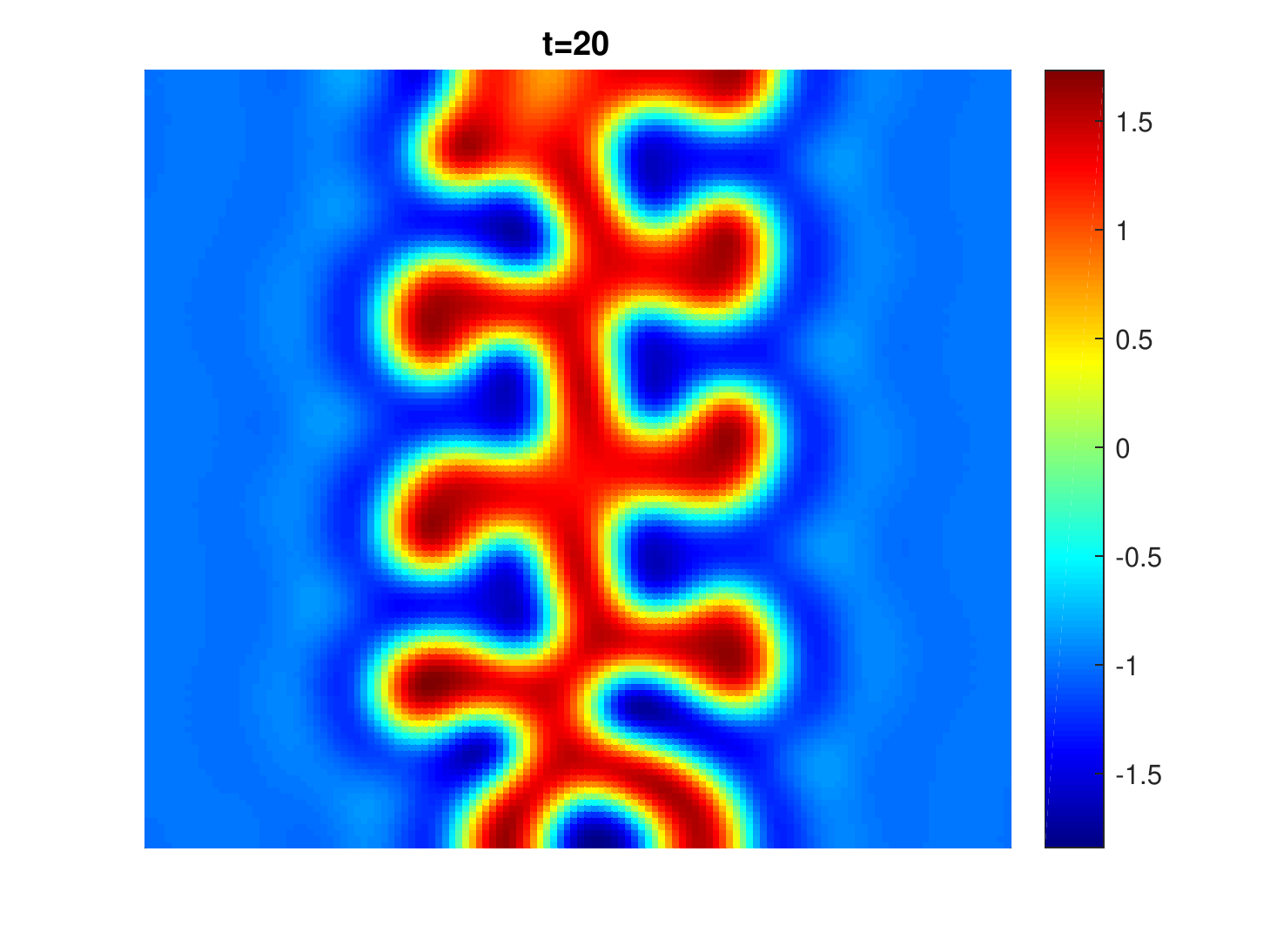}}
 \subfigure{\includegraphics[width=0.325\textwidth]{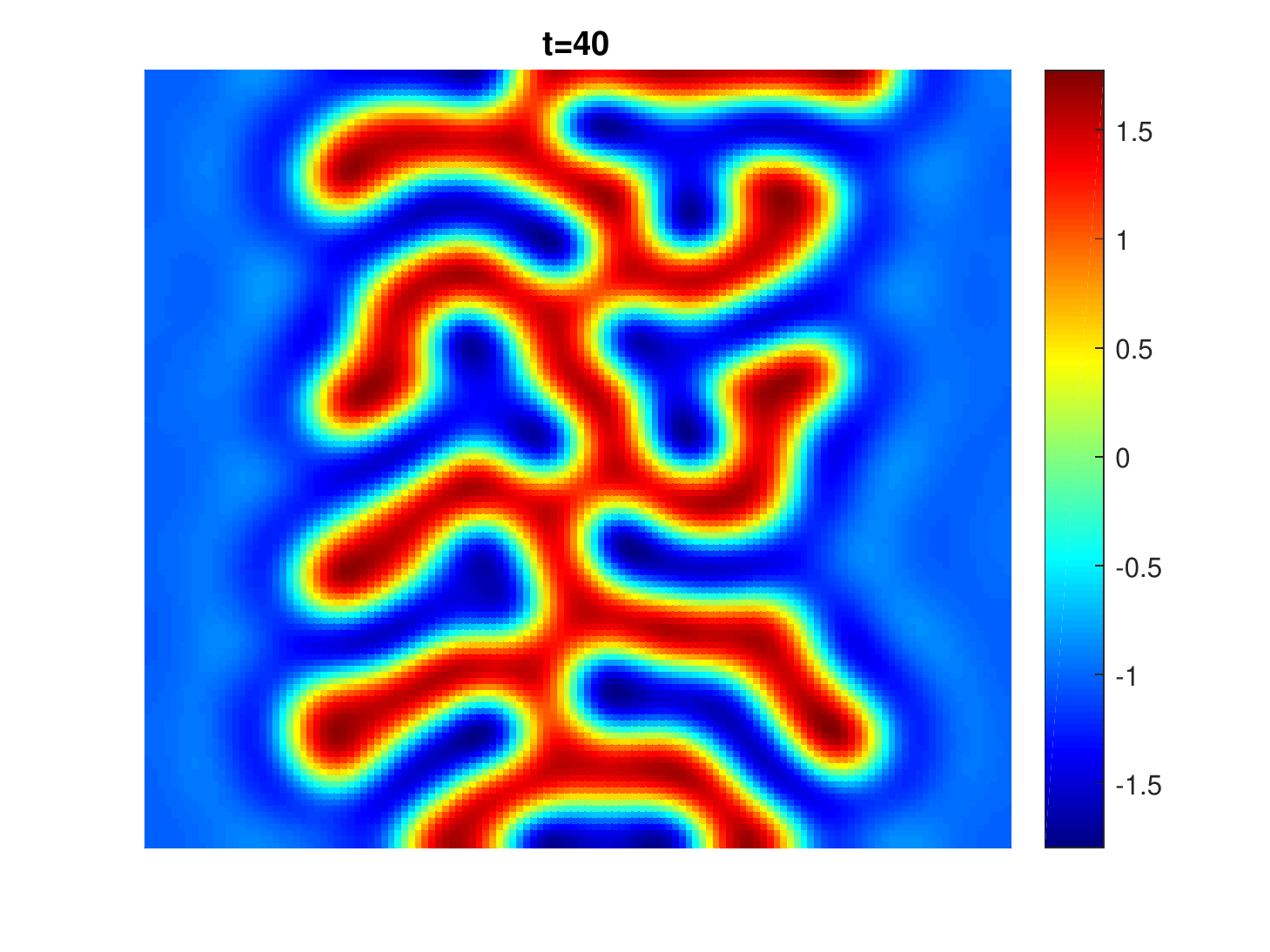}}
 \subfigure{\includegraphics[width=0.325\textwidth]{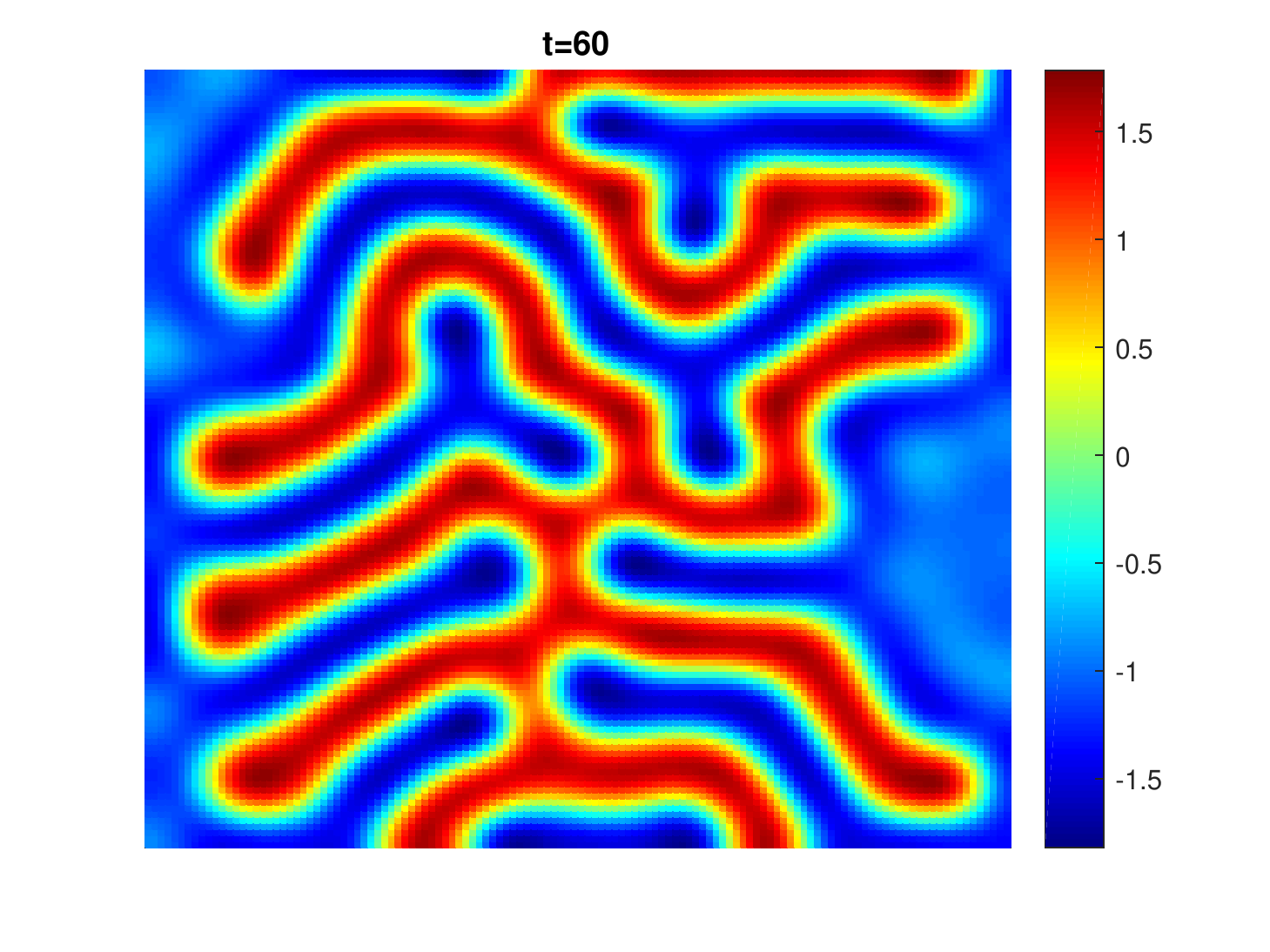}}
  \caption{ Evolution of patterns. 
  } \label{PatBifur}
 \end{figure}

 \begin{figure}
 \centering
 \subfigure{\includegraphics[width=0.49\textwidth]{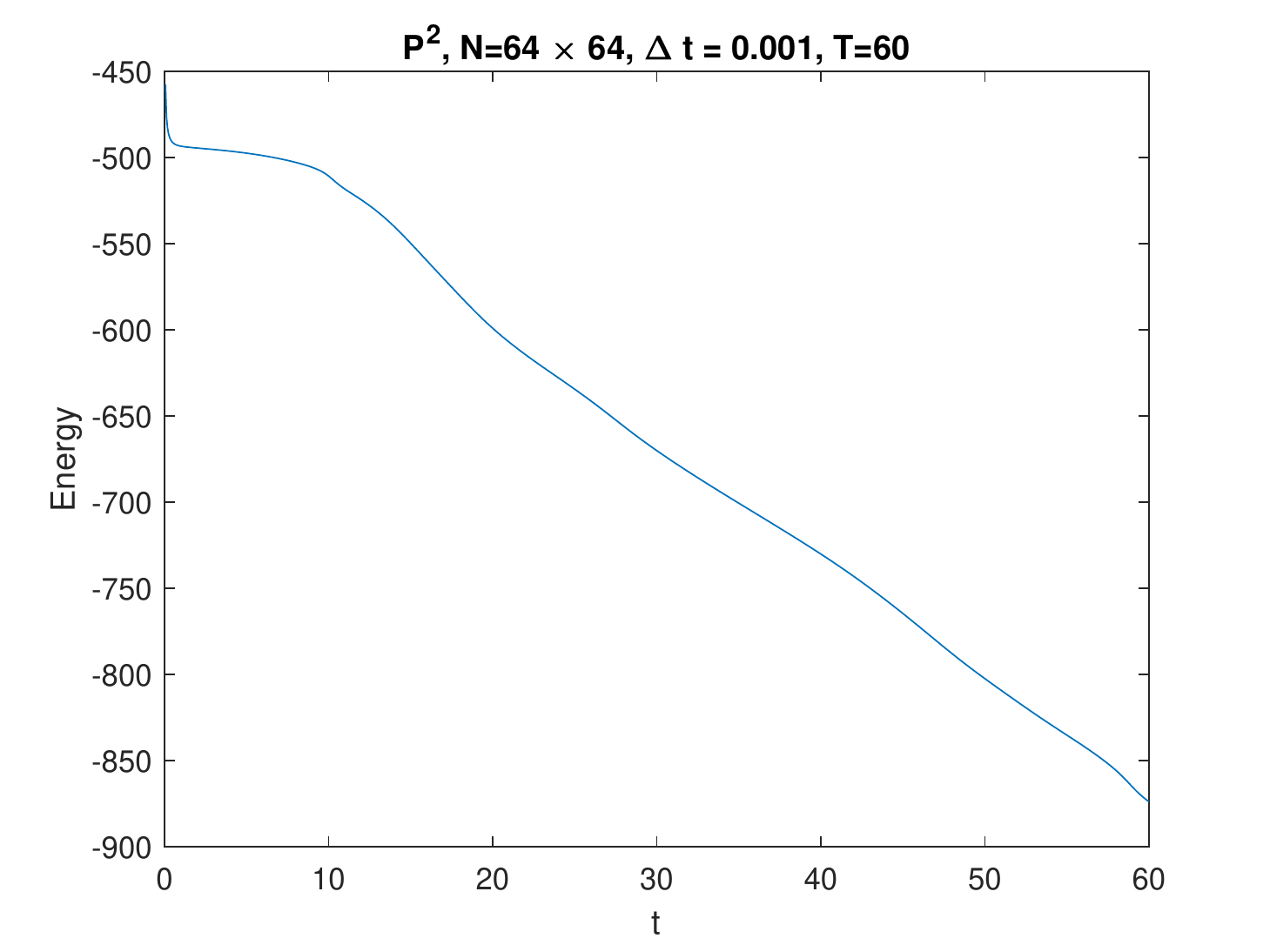}}
 \caption{ Energy evolution dissipation.
 } \label{BifEng}
 \end{figure}

\end{example}

\begin{example}\label{Ex2dPatt2} (Rolls and Hexagons) In this example, we test the formation and evolution of patterns that arise in the Rayleigh-B\'{e}nard convection by simulating with the Swift--Hohenberg equation (\ref{SH}) on rectangular domain $\Omega=[0,100]^2$,  subject to random initial data and periodic boundary conditions.  We apply scheme  (\ref{FPDGFull+}) based on $P^2$ polynomials using mesh $128 \times 128$ and time step size $\Delta t =0.01$.  Model parameters will be specified  below for different cases, and these choices of parameters have been used in \cite{PCC14, DA17}.

\noindent\textbf{Test case 1.} (Rolls)
The numerical solutions with parameters  $\varepsilon=0.3, \ g=0$ are shown in Figure \ref{PatBifur2},  from which we see periodic rolls for different times. We observe that the pattern evolves approaching the steady-state after $t>60$, as also evidenced by the energy evolution plot in Figure \ref{BifEng2}.
%

 \begin{figure}
 \centering
 \subfigure{\includegraphics[width=0.325\textwidth]{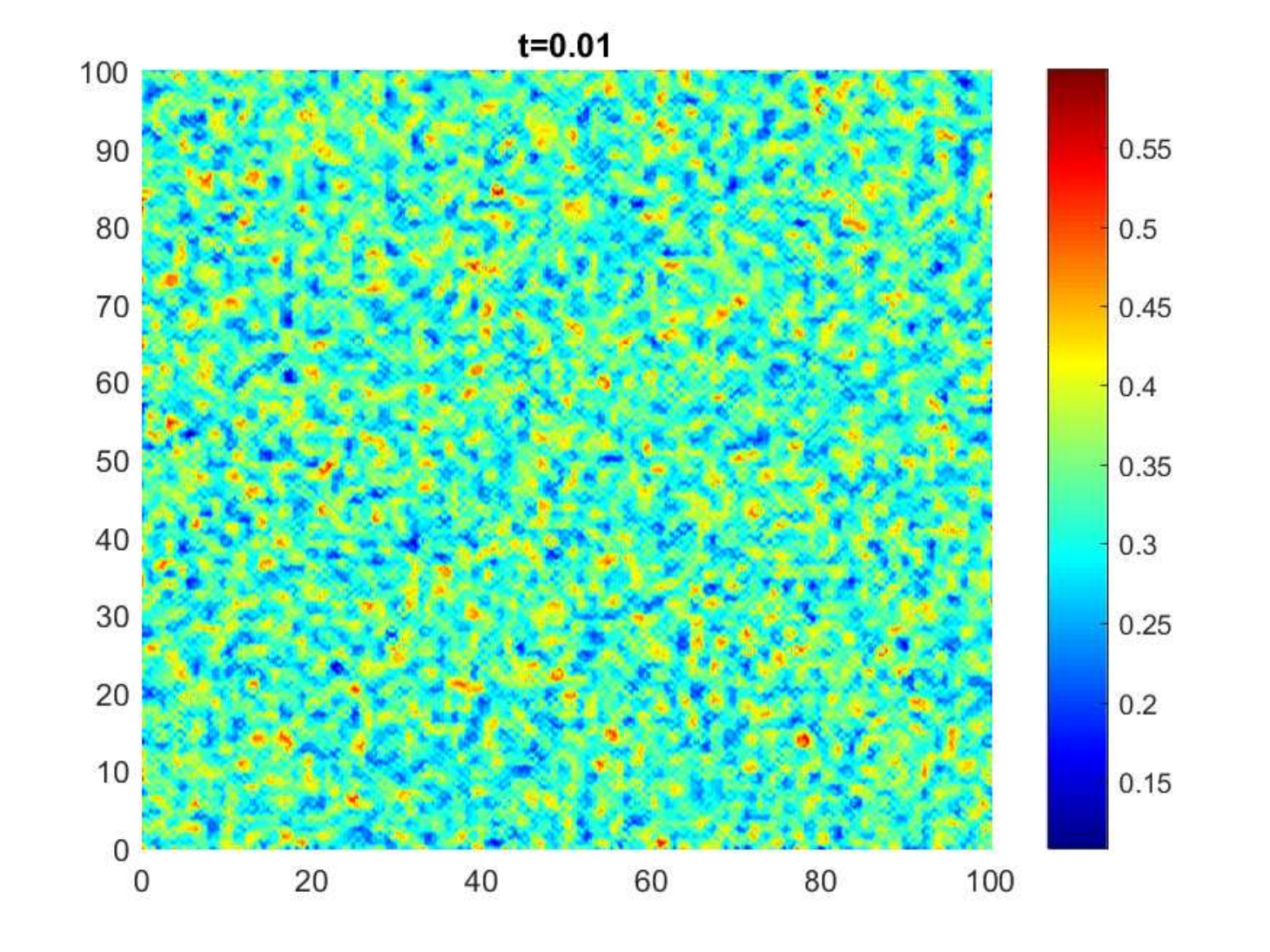}}
 \subfigure{\includegraphics[width=0.325\textwidth]{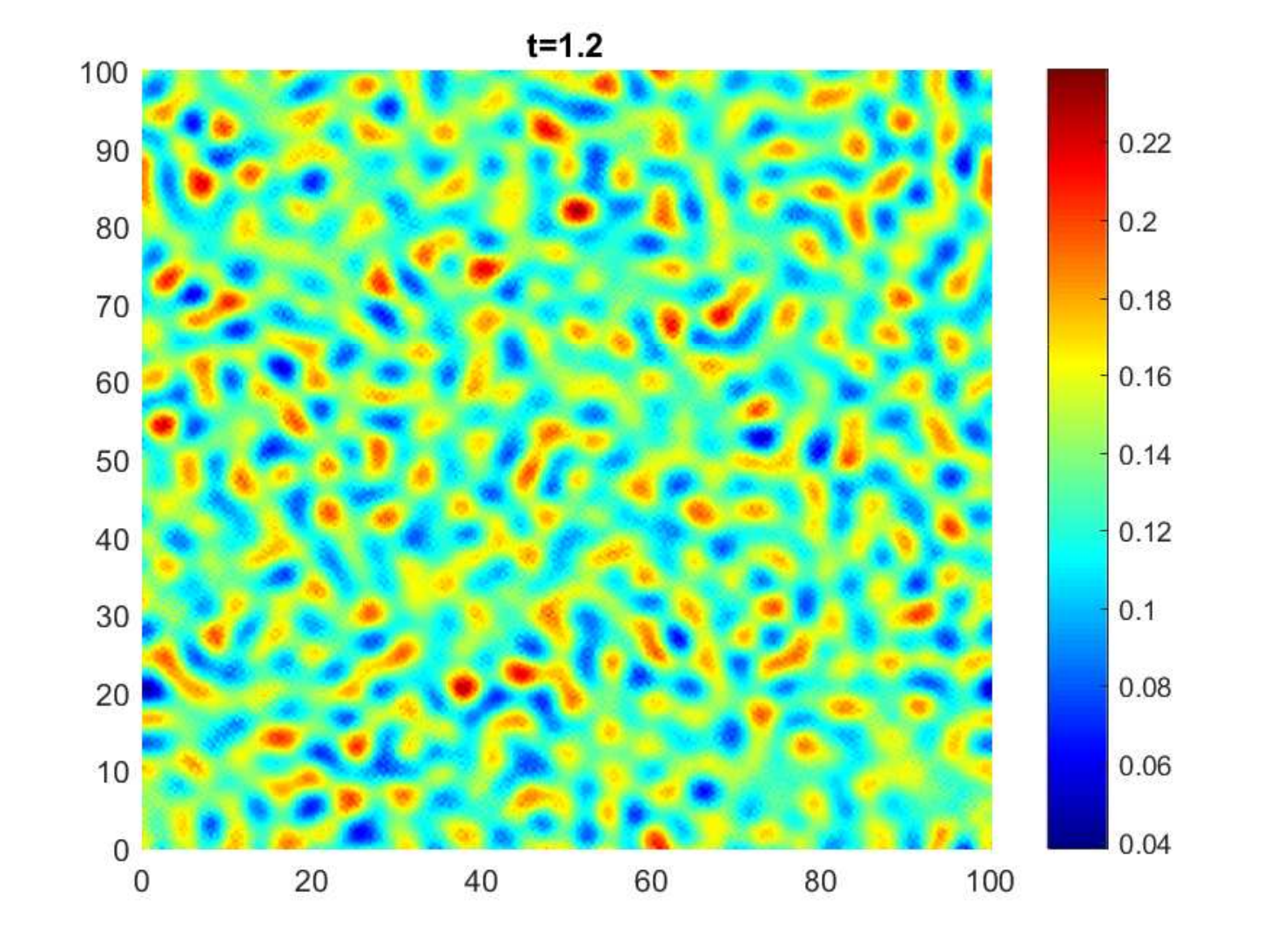}}
 \subfigure{\includegraphics[width=0.325\textwidth]{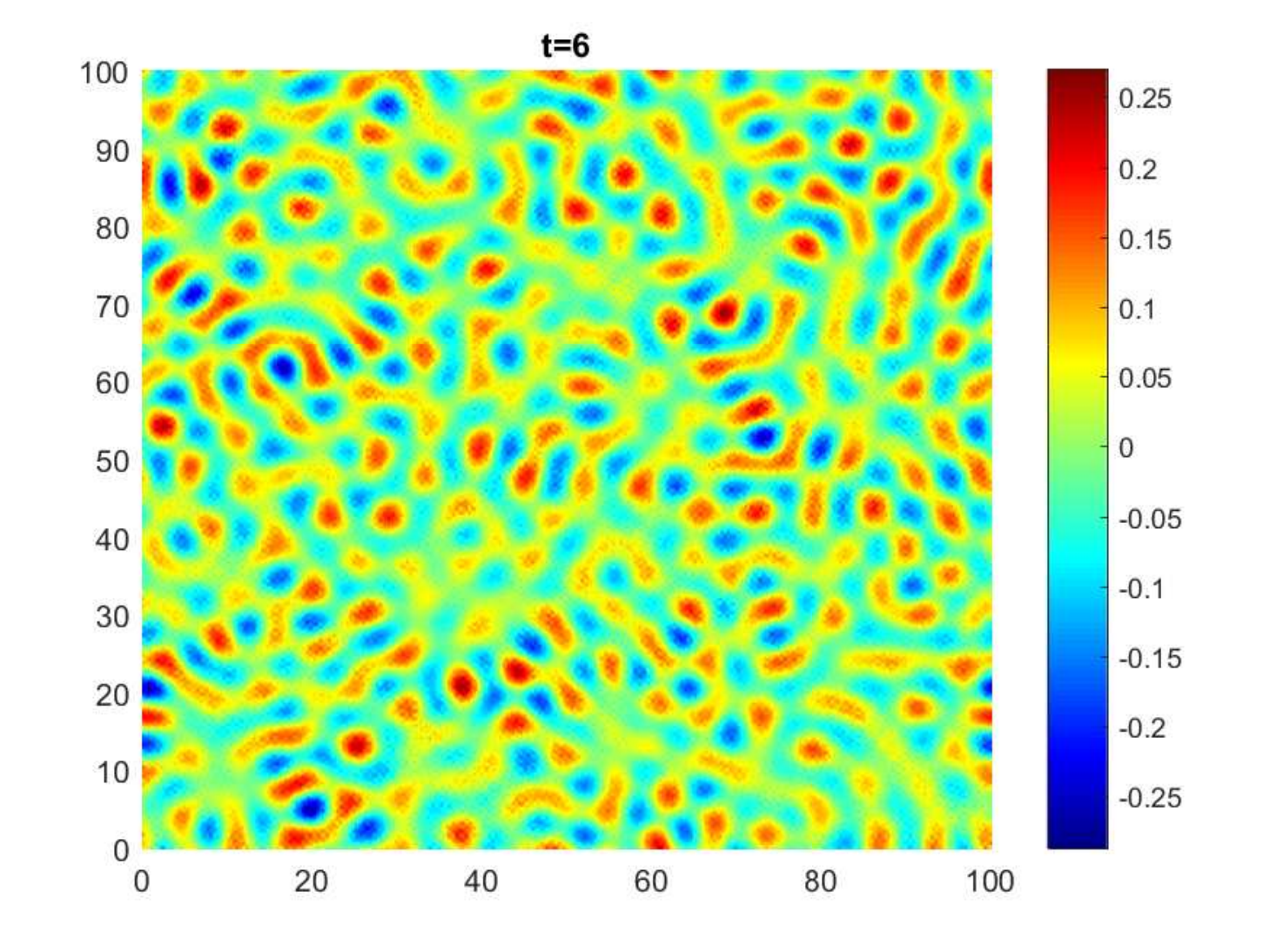}}
 \subfigure{\includegraphics[width=0.325\textwidth]{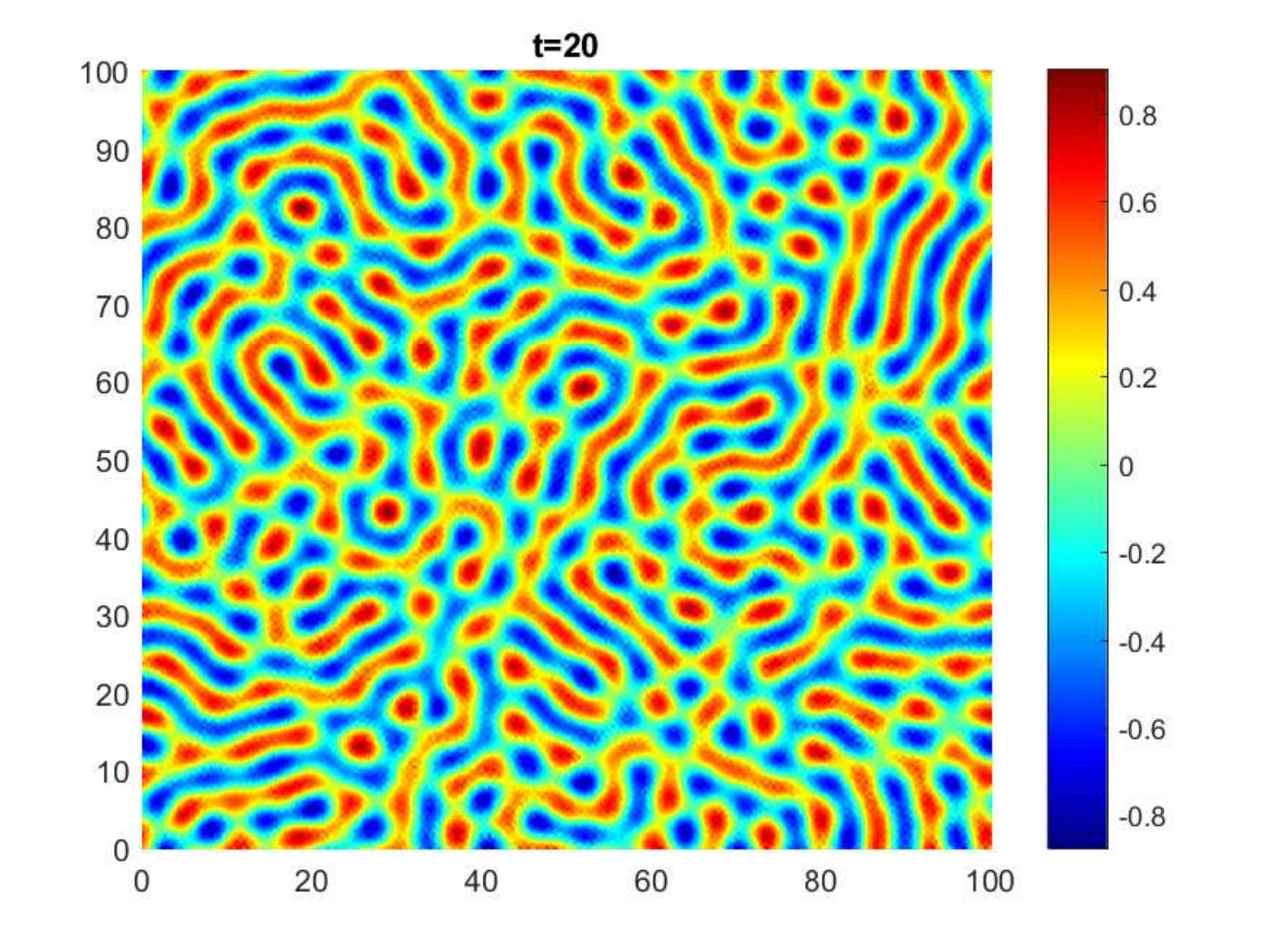}}
 \subfigure{\includegraphics[width=0.325\textwidth]{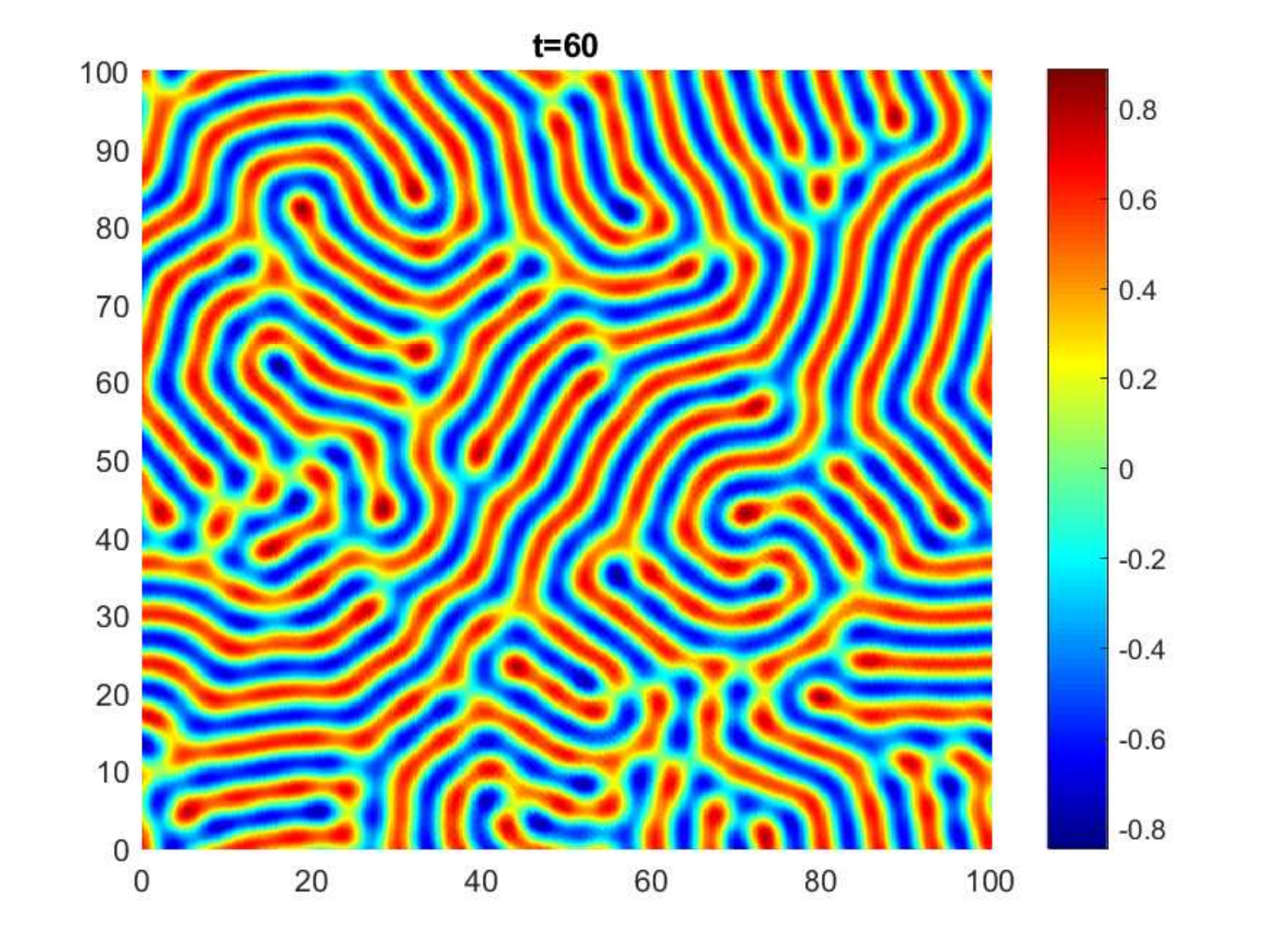}}
 \subfigure{\includegraphics[width=0.325\textwidth]{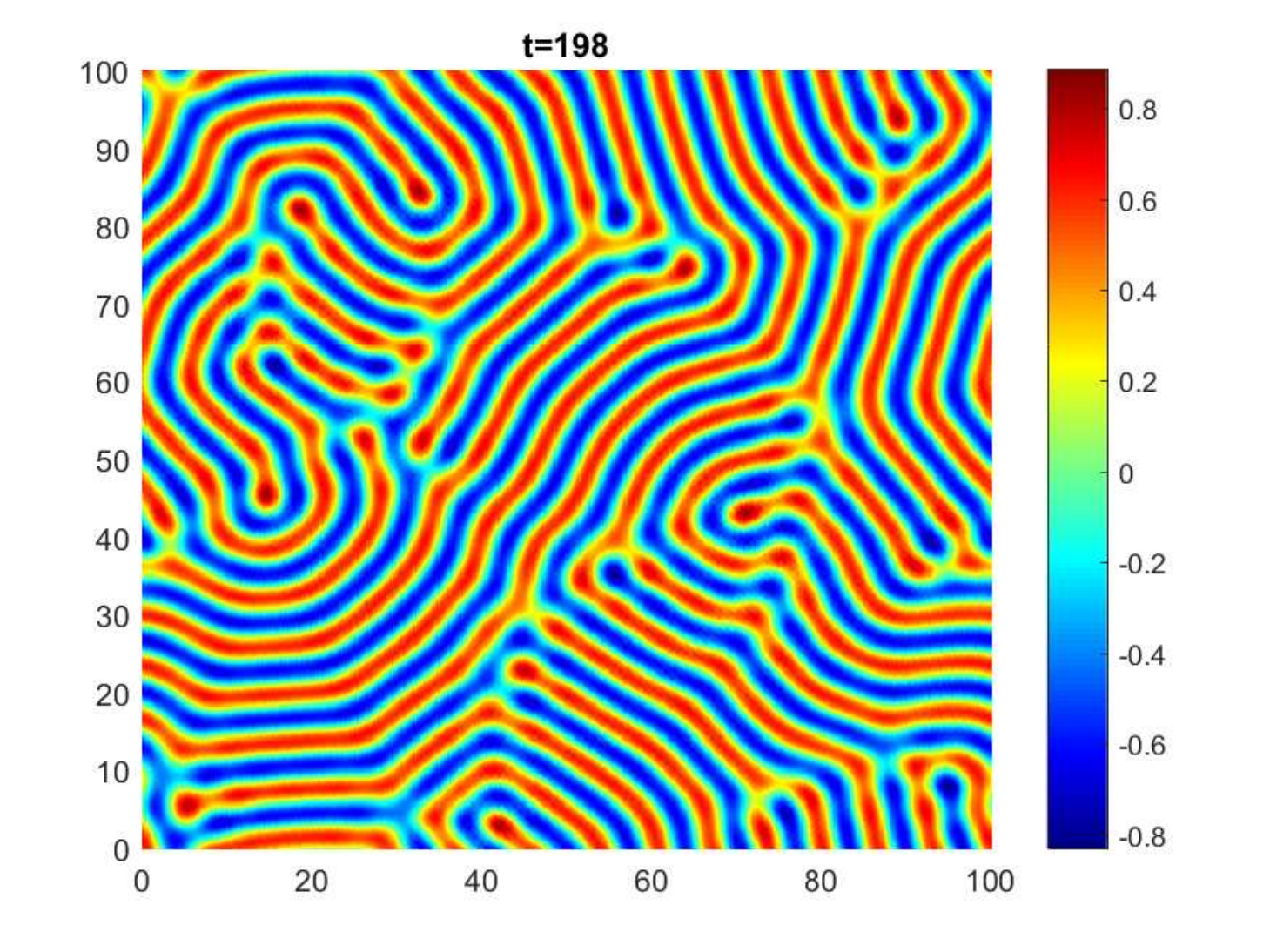}}
  \caption{ Evolution of periodic rolls.
  } \label{PatBifur2}
 \end{figure}

 \begin{figure}
 \centering
 \subfigure{\includegraphics[width=0.49\textwidth]{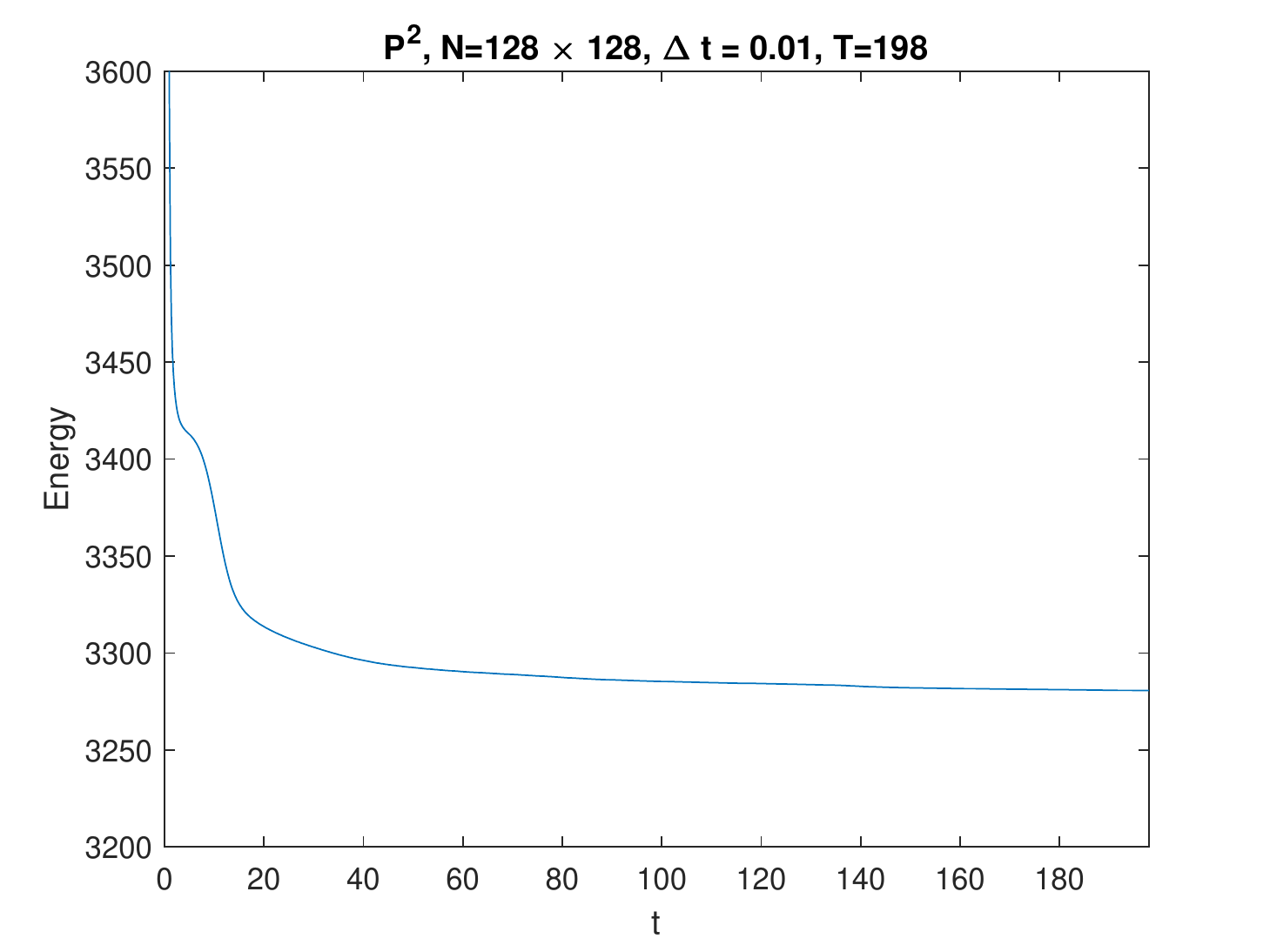}}
 \caption{ Energy evolution dissipation.
 } \label{BifEng2}
 \end{figure}

\noindent\textbf{Test case 2.} (Hexagons) The numerical solutions with $\varepsilon=0.1, \ g=1.0$ are reported in Figure \ref{PatBifur3}, while the snapshots from $t=0$ to $t=198$ reveal vividly the formation and evolution of the hexagonal pattern. The pattern evolution looks slow in the beginning, similar to that of rolls as shown in Figure \ref{PatBifur2}.  However, we observe that at a certain point, before $t=20$ in this case, lines break up giving way to single droplets that take hexagonal symmetry, as also observed  in \cite{PCC14, DA17}. A stable hexagonal pattern is taking its shape  after $t\geq 40$, and the steady state is approached. The energy evolution in Figure \ref{BifEng3} clearly confirms this.

  \begin{figure}
 \centering
 \subfigure{\includegraphics[width=0.325\textwidth]{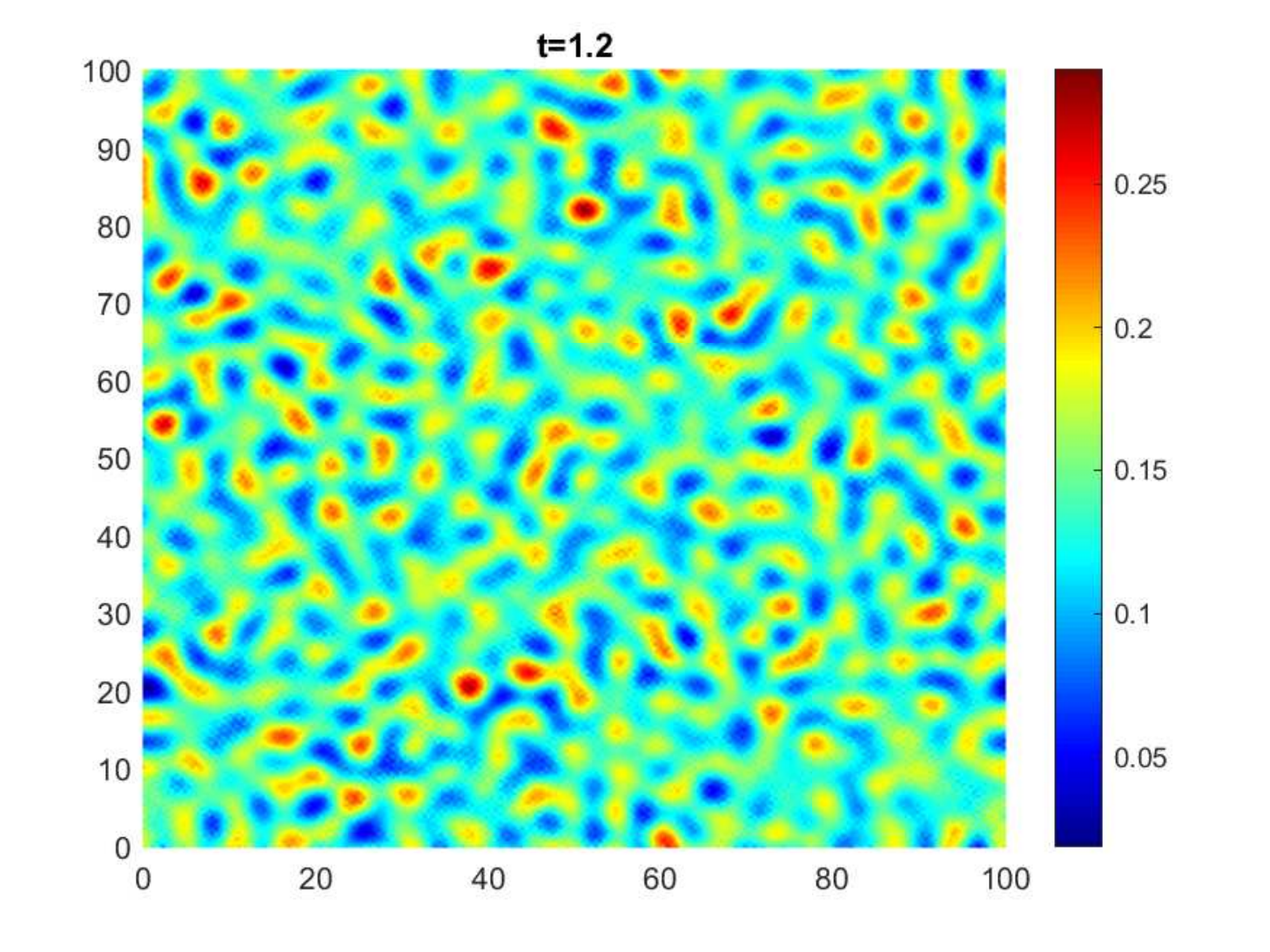}}
 \subfigure{\includegraphics[width=0.325\textwidth]{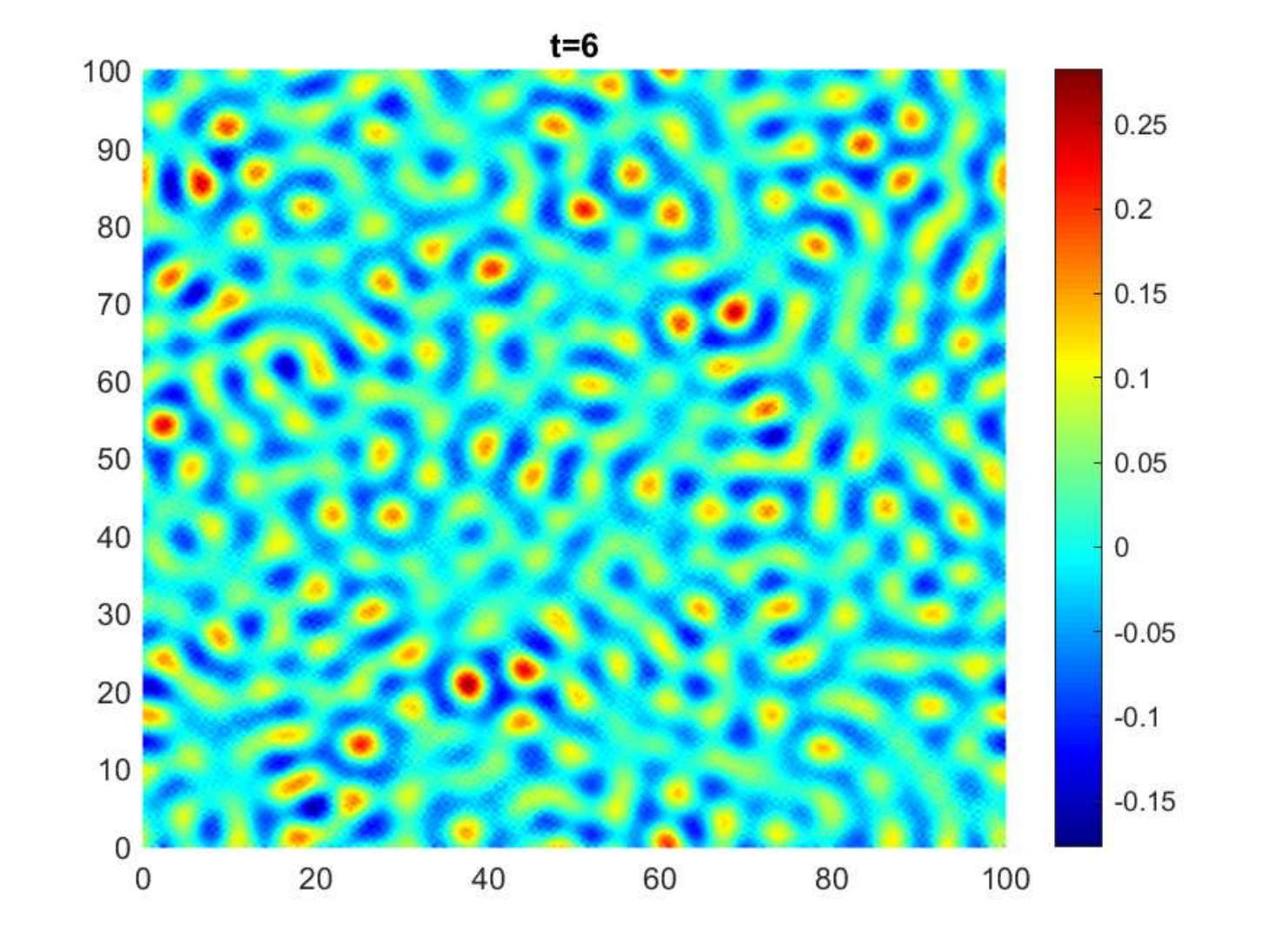}}
 \subfigure{\includegraphics[width=0.325\textwidth]{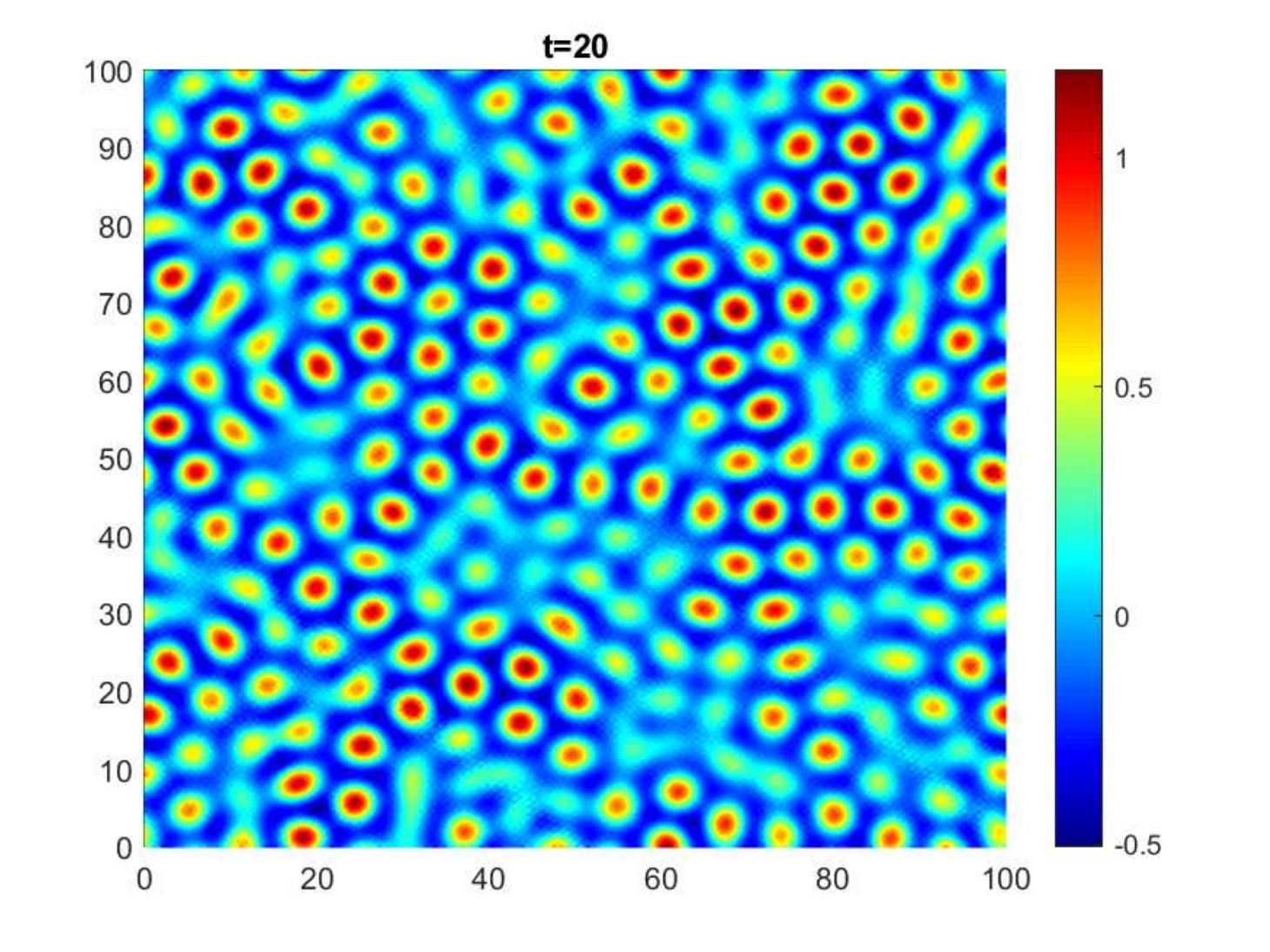}}
 \subfigure{\includegraphics[width=0.325\textwidth]{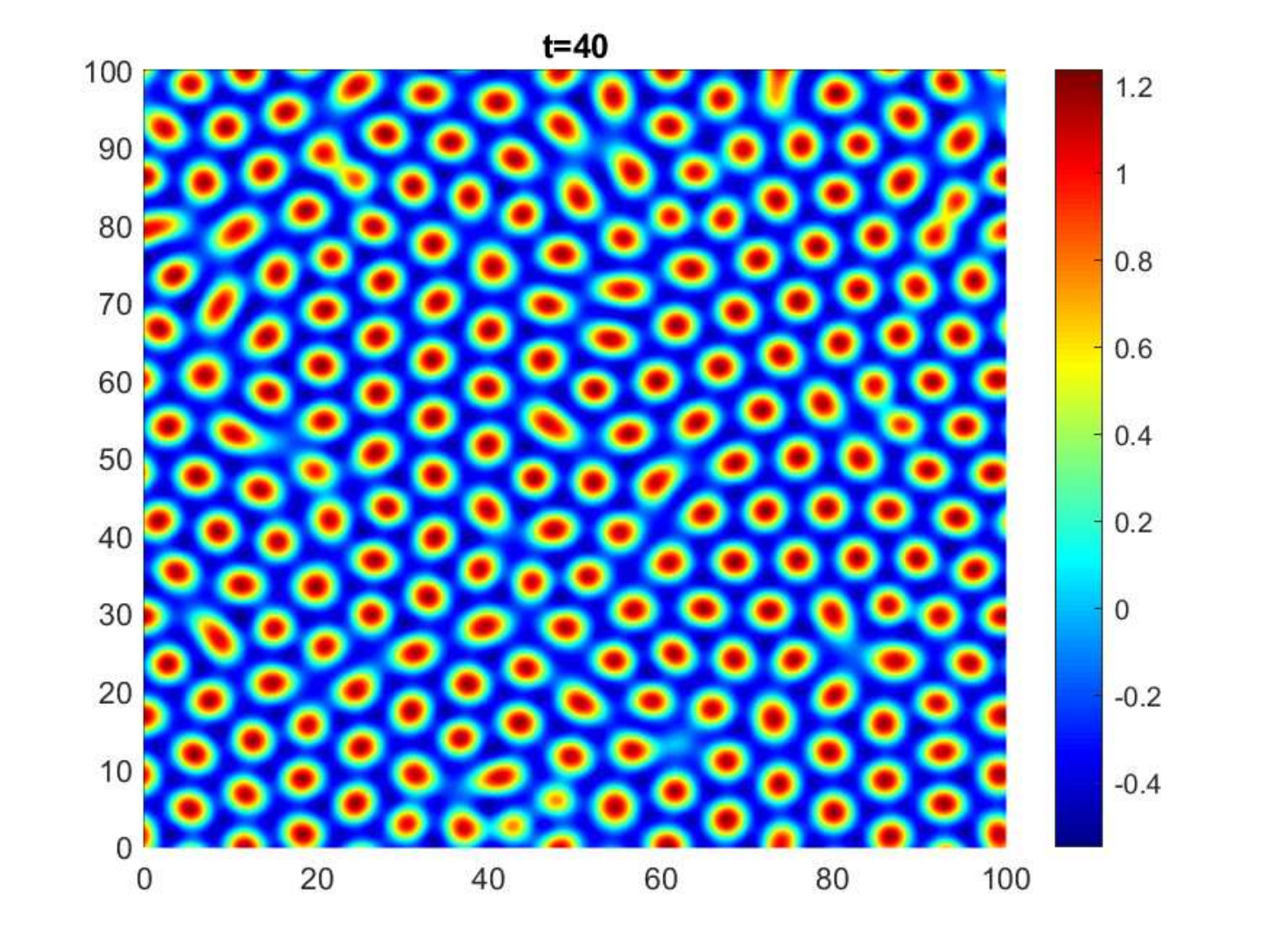}}
 \subfigure{\includegraphics[width=0.325\textwidth]{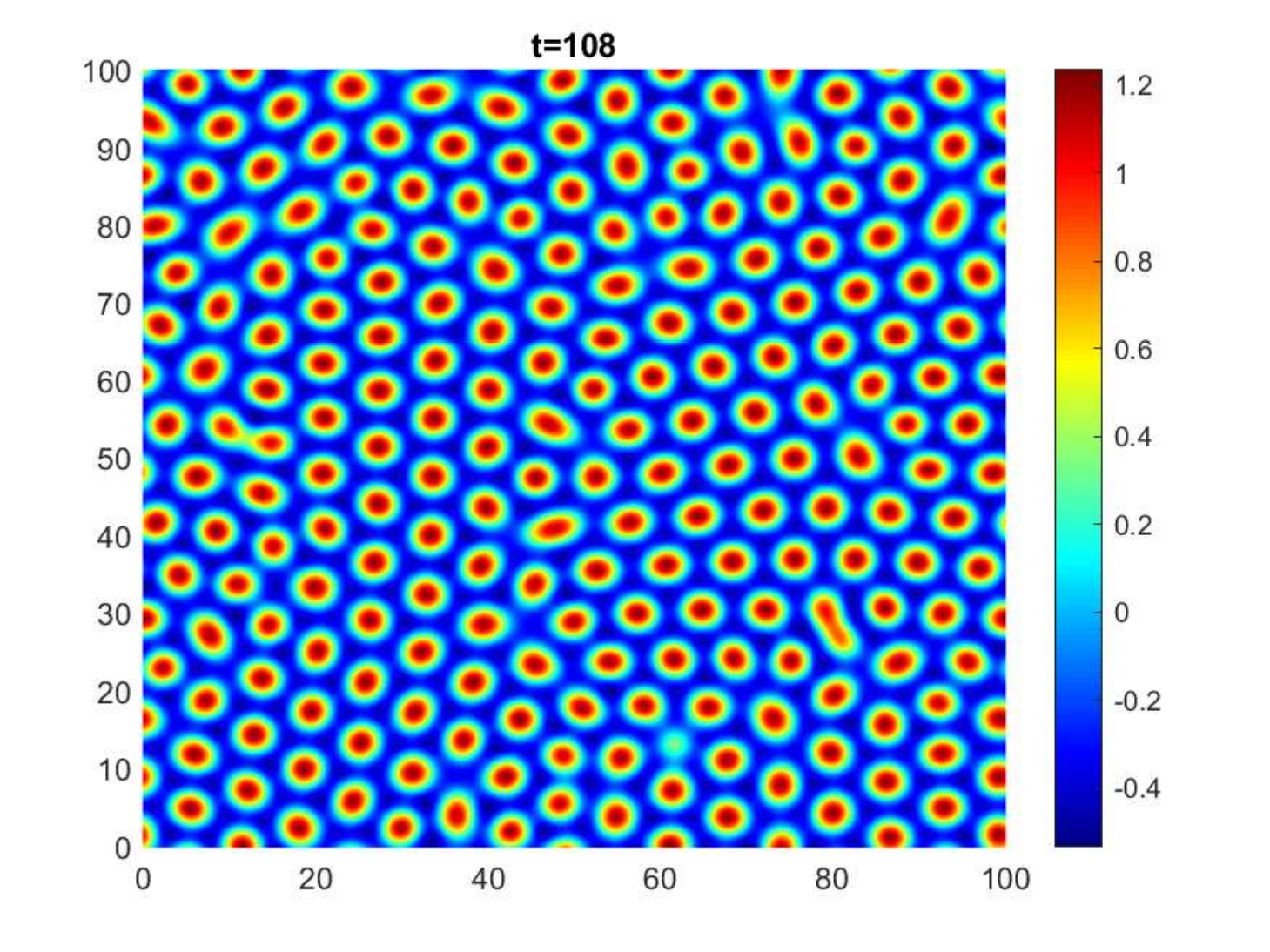}}
 \subfigure{\includegraphics[width=0.325\textwidth]{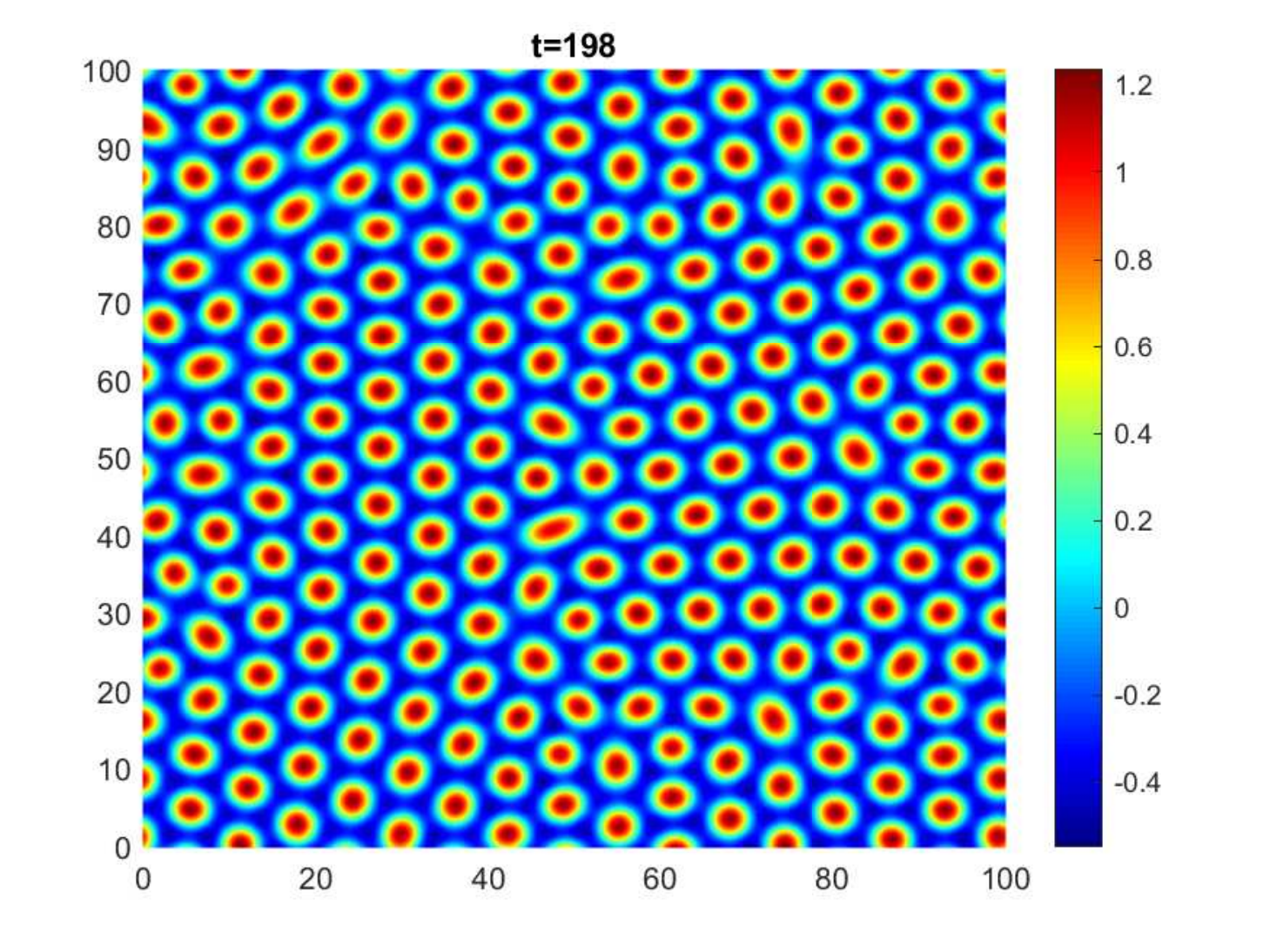}}
  \caption{ Evolution of hexagonal patterns. 
  } \label{PatBifur3}
 \end{figure}

 \begin{figure}
 \centering
 \subfigure{\includegraphics[width=0.49\textwidth]{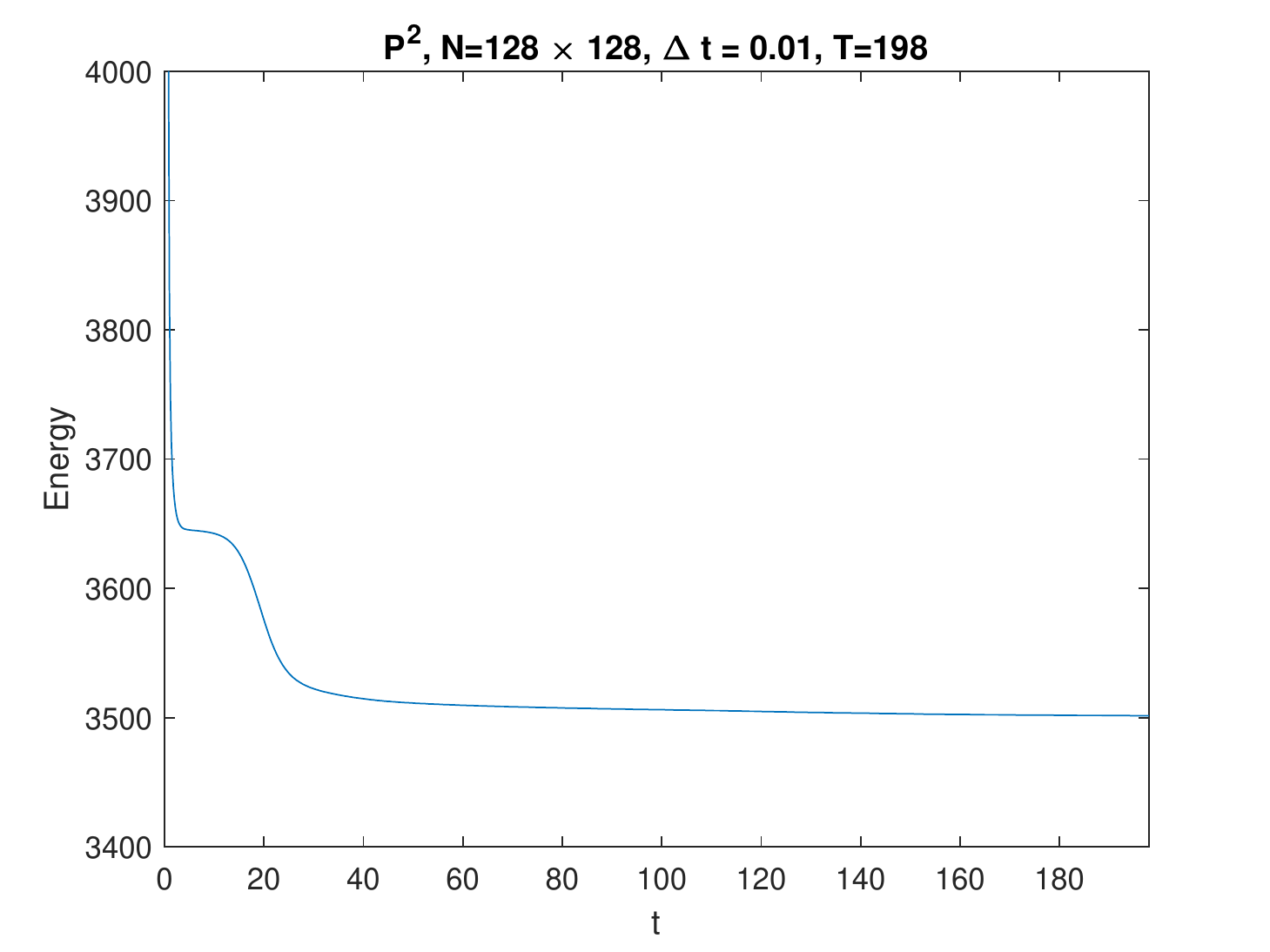}}
 \caption{ Energy evolution dissipation.
 } \label{BifEng3}
 \end{figure}

\end{example}

\begin{example}\label{varmethods} This example is to compare the numerical performance of three different time discretization techniques when applied to our mixed DG method (see also \cite{LY18} for details in its semi-discrete formulation), including \\
(i) 
 the second order IEQ-DG scheme  (\ref{FPDGFull+}); \\
(ii) 
the DG scheme (\ref{FPDGFullNon}), which was  introduced in  \cite{LY18};  and \\
(iii) the second order time discretization in \cite{GN12}, for which one  finds  $(u^{n}_h, q_h^{n}) \in V_h  \times V_h$ such that
\begin{subequations}\label{FPDGFullNon2}
\begin{align}
   \left(  \frac{u_h^{n+1} - u_h^n}{\Delta t}, \phi \right) = & - A(q_h^{n+1/2},\phi)- \left( \frac{1}{2}\left( \Phi'(u_h^{n+1})+\Phi'(u_h^{n}) \right) - \frac{(u_h^{n+1} - u_h^n)^2}{12}\Phi'''(u_h^{n}),\phi\right)\\
    (q_h^{n}, \psi) = & A(u_h^{n},\psi),
\end{align}
\end{subequations}
for all $\phi, \ \psi \in V_h$.

Though all three satisfy certain energy dissipation law,  (ii) and (iii) have to be solved by appropriate iterative techniques. We recall that  for the SH equation (\ref{SH}),
$$
\Phi(u)=-\frac{\epsilon}{2}u^2-\frac{g}{3}u^3+\frac{u^4}{4}.
$$
The iterative scheme used in \cite{LY18} for (\ref{FPDGFullNon}) is
the following
\bq\label{FPDGFull1DNon}
\ba
     \left( \frac{u_h^{n+1,l+1}-u_h^n}{\Delta t}, \phi \right) & +
     \frac{1}{2}A(q_h^{n+1,l+1},\phi)=
    -\frac{1}{2}A(q_h^{n},\phi)\\
    & -\left( G_1(u_h^{n+1,l},u_h^n)u_h^{n+1,l+1}+G_2(u_h^{n+1,l},u_h^n), \phi \right),\\
   \frac{1}{2}A(u_h^{n+1,l+1},\psi) & -\frac{1}{2}(q_h^{n+1,l+1}, \psi) = 0,
\ea
\eq
where $G_1(u_h^{n+1,0},u_h^n)=G_1(u_h^n,u_h^n)$, the iteration stops as $\|u_h^{n+1,l}-u_h^{n+1,l-1}\| < \eta$ for certain $l=L \ (L \geq 1)$ and some tolerance $\eta>0$.  Then we update by setting $ u_h^{n+1} =u_h^{n+1,L}.$ Here
\bqs
\ba
G_1(w,v)=& -\frac{\varepsilon}{2}-\frac{g}{3}(w+v)+\frac{1}{4}(w^2+wv+v^2), \\
G_2(w,v)=& -\frac{\varepsilon}{2}v-\frac{g}{3}v^2+\frac{1}{4}v^3.
\ea
\eqs
The scheme (\ref{FPDGFullNon2}) can still be solved iteratively by (\ref{FPDGFull1DNon}) if one can decompose the nonlinear term in (\ref{FPDGFullNon2}a) as
$$
\frac{1}{2}\left( \Phi'(u_h^{n+1})+\Phi'(u_h^{n}) \right) - \frac{(u_h^{n+1} - u_h^n)^2}{12}\Phi'''(u_h^{n})=G_1(u_h^{n+1}, u_h^n)u_h^{n+1}+G_2(u_h^{n+1}, u_h^n).
$$
We consider two decompositions:
\bq\label{iter1}
\ba
G_1(w,v)=& \frac{1}{2}(-\varepsilon-gw-w^2)-\frac{1}{12}(w-2v)(6v-2g),\\
G_2(w,v)= & \frac{1}{2}\Phi'(v)-\frac{1}{12}v^2(6v-2g),
\ea
\eq
and
\bq\label{iter2}
\ba
G_1(w,v)=& \frac{1}{2}(-\varepsilon-gw-3w^2)-\frac{1}{12}(w-2v)(6v-2g),\\
G_2(w,v)= & \frac{1}{2}\Phi'(v)-\frac{1}{12}v^2(6v-2g)-w^3,
\ea
\eq
\noindent\textbf{Test case 1.}  We consider the SH equation (\ref{SH}) with a source
$$
f(x,y, t)=- \varepsilon v -gv^2+ v^3,
$$
where $v=e^{-49t/64}\sin(x/2)\sin(y/2)$, and parameters $\varepsilon=0.025, g=0.05$. For initial data (\ref{initex12}), and
boundary condition $\partial_\nu u = \partial_\nu \Delta u = 0, \ (x,y) \in \partial \Omega$, where domain is $\Omega=[-2\pi, 2\pi]^2$,
we have an exact solution given by
\bqs
u(x,y,t) = e^{-49t/64}\sin(x/4)\sin(y/4), \quad (x, y) \in \Omega.
\eqs
We test  schemes (i)-(iii) based on $P^2$ polynomials with
$$
\frac{1}{2}\left(f(\cdot, t^{n+1}, \phi)+f(\cdot, t^{n}, \phi)\right),
$$
added to the right hand side of (\ref{FPDGFull+}c), (\ref{FPDGFullNon}a) and (\ref{FPDGFullNon2}a), respectively. For (ii) and (iii), we take the tolerance $\eta=10^{-12}$.

We compute the numerical solution at $T=2$ with mesh size $32 \times 32$ and time steps $\Delta t=2^{-m}$ for $2\leq m\leq 5$, the $L^2, L^\infty$ errors and orders of convergence in time are shown in Table \ref{timeacc2}, and these results show that schemes (i)-(iii) are all of second order accuracy in time.

\begin{table}[!htbp]\tabcolsep0.03in
\caption{$L^2, L^\infty$ errors and EOC at $T = 2$ with time step $\Delta t$.}
\begin{tabular}[c]{||c|c|c|c|c|c|c|c|c||}
\hline
\multirow{2}{*}{Method} &  \multirow{2}{*}{ } & $\Delta t=2^{-2}$ & \multicolumn{2}{|c|}{$\Delta t=2^{-3}$} & \multicolumn{2}{|c|}{$\Delta t=2^{-4}$} & \multicolumn{2}{|c||}{$\Delta t=2^{-5}$}  \\
\cline{3-9}
& & error & error & order & error & order & error & order\\
\hline
\multirow{2}{*}{(i)}  & $\|u-u_h\|_{L^2}$ &  1.58904e-02 & 3.28568e-03 & 2.27 & 7.79139e-04 & 2.08 & 1.88606e-04 & 2.05  \\
\cline{2-9}
 & $\|u-u_h\|_{L^\infty}$  & 2.86144e-03 & 6.04098e-04 & 2.24 & 1.59953e-04 & 1.92 & 4.25000e-05 & 1.91  \\
\hline
\hline
\multirow{2}{*}{(ii)}  & $\|u-u_h\|_{L^2}$ & 1.21293e-02 & 3.01853e-03 & 2.01 & 7.62040e-04 & 1.99 & 1.89204e-04 & 2.01  \\
\cline{2-9}
 & $\|u-u_h\|_{L^\infty}$  & 2.75613e-03 & 6.82343e-04 & 2.01 & 1.72912e-04 & 1.98 & 4.27627e-05 & 2.02  \\
 \hline
 \hline
\multirow{2}{*}{(iii)} & $\|u-u_h\|_{L^2}$ & 1.15070e-02 & 2.94199e-03 & 1.97 & 7.52614e-04 & 1.97 & 1.88020e-04 & 2.00  \\
\cline{2-9}
 & $\|u-u_h\|_{L^\infty}$  & 2.57787e-03 & 6.60569e-04 & 1.96 & 1.70270e-04 & 1.96 & 4.24274e-05 & 2.00  \\
\hline
\end{tabular}\label{timeacc2}
\end{table}

\noindent\textbf{Test case 2.} We attempt to recover the pattern observed in Example \ref{Ex2dPatt} at $T=10$ by using schemes (i)-(iii) with meshes $64\times 64$ and time steps $\Delta t=2^{-m}$ for $2\leq m\leq 7$. For scheme (i), we take $B=10^4$ since  we observe that  larger $B$ can give better approximation, such effect seems visible only for larger $\Delta t$.  For both (ii) and (iii),  we take the tolerance $\eta=10^{-10}$, and
use the same preconditioner and solver as for (i).

For schemes  (i)-(iii) both the maximum number of iterations at each time step  and the total CPU time from $t=0$ to $t=T$ are presented in Table \ref{tab2daccNeu32};  the CPU time is highlighted when the expected pattern is observed. The results show that scheme (i) uses the least number of iterations and the least CPU time to obtain the expected pattern, and hence the most efficient one among three schemes.

\begin{table}[!htbp]\tabcolsep0.03in
\caption{Iterations and CPU time in seconds at $T=10$ with meshes $64 \times 64$.}
\begin{tabular}[c]{||c|c|c|c|c|c|c|c||}
\hline
Method & $\Delta t$ & $2^{-2}$ & $2^{-3}$ & $2^{-4}$ & $2^{-5}$ & $2^{-6}$ & $2^{-7}$\\
\hline
\hline
\multirow{2}{*}{(i)}   & Iterations &  1 & 1 & 1 & 1 & 1 & 1 \\
\cline{2-8}
 & CPU time  & 842 & 1128 & 1557 & 2320 & \textbf{3717} & \textbf{6042} \\
\hline
\hline
\multirow{2}{*}{(ii)}  & Iterations & 20 & 13 & 10 & 8 & 7 & 6\\
\cline{2-8}
 & CPU time  & 7874 & 7024 & 7774 & \textbf{9818} & \textbf{13652} & \textbf{20542} \\
 \hline
\hline
\multirow{2}{*}{(iii)-(\ref{iter1})}  & Iterations & 18 & 12 & 9  & 8 & 7  &  6  \\
\cline{2-8}
 & CPU time   & 6478 & 6229 & 7296 & \textbf{9587} & \textbf{13497}  &  \textbf{20383} \\
\hline
\hline
\multirow{2}{*}{(iii)-(\ref{iter2})}  & Iterations & 13 & 11  & 9 &  7 & 7 & 6 \\
\cline{2-8}
 & CPU time   &  5748 & 6223 & 7595  & \textbf{9673} & \textbf{13526} & \textbf{20483} \\
\hline
\end{tabular}\label{tab2daccNeu32}
\end{table}

\end{example}

\section{Concluding remarks}
The Swift--Hohenberg equation is a higher-order nonlinear partial differential equation endowed with a gradient flow structure.
We proposed fully discrete discontinuous Galerkin (DG) schemes that inherit the nonlinear stability relationship of the continuous equation irrespectively of the mesh and time step sizes.  The spatial discretization is  based on the mixed DG method  introduced by us in \cite{LY18},  and the temporal discretization is based on \emph{Invariant Energy Quadratization} (IEQ) approach introduced in \cite{Y16} for the nonlinear potential. Coupled with a proper projection, the resulting IEQ-DG algorithm is explicit without resorting to any iteration method,  and proven to be unconditionally energy stable.  We present several numerical examples to assess the performance of the schemes in terms of accuracy and energy stability. The numerical results on two dimensional pattern formation problems indicate that the method is able to deliver comparable patterns of high accuracy.

Pattern formation is the result of self-organization systems and there are many examples of this phenomenon, in spite of the different mechanisms
that trigger and amplify the instability.  The present method should be applicable to a wide variety of processes and can be variationally improved if necessary.

\section*{Acknowledgments}
 This research was supported by the National Science Foundation under Grant DMS1812666 and by NSF Grant RNMS
(KI-Net) 1107291.

\bigskip

\end{document}